\documentclass[12pt]{article}

\usepackage[top=1.5cm,bottom=1.5cm,left=1.5cm,right=1.5cm]{geometry}

\usepackage[utf8]{inputenc}
\usepackage[english]{babel}

\usepackage{amsfonts,amssymb,amsthm,amsmath}

\usepackage{lmodern}
\usepackage{microtype}

\usepackage{epigraph}

\usepackage{graphicx,xcolor}

\usepackage[pdfencoding=unicode, psdextra, citecolor=blue, urlcolor=blue,
	linkcolor=blue, colorlinks=true, bookmarksopen=true]{hyperref}

\usepackage{comment}

\usepackage{xspace}

\newcommand{\R}{\mathbb{R}}
\newcommand{\Z}{\mathbb{Z}}
\newcommand{\K}{\mathcal{K}}
\newcommand{\cl}{\mathrm{cl}\,}
\newcommand{\const}{\mathrm{const}}
\newcommand{\eqdef}{\stackrel{\mathrm{def}}{=}}
\newcommand{\I}{\mathrm{i}\mkern1mu}
\newcommand{\ddt}[1][]{
	\frac{d^{#1}}{dt^{#1}}
}
\newcommand{\cg}{($\mathfrak{G}$)\xspace}
\newcommand{\cw}{($\mathfrak{W}$)\xspace}

\newtheorem{theorem}{Theorem}
\newtheorem{corollary}{Corollary}
\newtheorem{proposition}{Proposition}
\newtheorem{lemma}{Lemma}

\theoremstyle{remark}

\theoremstyle{definition}
\newtheorem{definition}{Definition}

\usepackage{authblk}

\title{Derivatives of Sub-Riemannian Geodesics are $L_p$-H\"older Continuous}

\author[1]{Lev Lokutsievskiy\footnote{This work was performed at the Steklov International Mathematical Center and supported by the Ministry of Science and Higher Education of the Russian Federation (agreement no. 075-15-2022-265).}}
\author[2]{Mikhail Zelikin}
\affil[1]{Steklov Mathematical Institute of Russian Academy of Sciences}
\affil[2]{Lomonosov Moscow State University}

\begin{document}

\maketitle

\epigraph{\it Mathematics consists in proving the most obvious thing in the least obvious way.}{\it George P\'olya}

\begin{abstract}

	This article is devoted to the long-standing problem on the smoothness of sub-Riemannian geodesics. We prove that the derivatives of sub-Riemannian geodesics are always $L_p$-H\"older continuous. Additionally, this result has several interesting implications. These include (i) the decay of Fourier coefficients on abnormal controls, (ii) the rate at which they can be approximated by smooth functions, (iii) a generalization of the Poincar\'e inequality, and (iv) a compact embedding of the set of shortest paths into the space of Bessel potentials.
		
\end{abstract}

\section{Introduction}

Nonholonomic and non-Riemannian problems are currently being extensively studied by numerous mathematicians. The interest in these problems started to rapidly grow in the mid-20th century, following the works of famous mathematicians such as E. Cartan, Carath\'eodory, Gromov, Montgomery, Mitchell, Vershik, Agrachev, and many others. The most mysterious and intriguing object in the theory is abnormal geodesics, which arise as critical points of the end-point map. The question of how to investigate these geodesics has been present since the inception of the theory, and in particular, the question on their smoothness has not yet been resolved.

In Riemannian geometry, analyzing geodesic smoothness properties is relatively straightforward. Riemannian geodesics satisfy the Euler--Lagrange equation and hence they are as smooth (or analytic) as the right-hand side of the equation. However, sub-Riemannian geodesics do not always satisfy the system of Euler--Lagrange equations. Instead, they satisfy a system of differential inclusions that comes from the Pontryagin maximum principle (PMP). A sub-Riemannian geodesic is called normal if it satisfies PMP with a non-degenerate Hamiltonian, and in this case, it satisfies the Hamiltonian version of the Euler--Lagrange equation and hence it has to be smooth. However, strictly abnormal geodesics do not satisfy PMP with a non-degenerate Hamiltonian. In this case, the corresponding system of differential inclusions is not equivalent to any system of ordinary differential equations, and hence the classical Riemannian approach on investigating geodesics smoothness suffers almost complete failure.

The Pontryagin maximum principle still holds some potential for studying abnormal geodesics. However, investigations based on the principle typically depend on the step value of the sub-Riemannian manifold, as they all rely on the Goh condition. For example, geodesics on sub-Riemannian manifolds of step $s=2$ are smooth because the Goh condition implies that there are no abnormal geodesics in this case (see~\cite{AgrachevBook}). Even in the case $s=3$, it is very difficult to study smoothness of abnormal geodesics using this method (though there are some results in this direction~\cite{TanYang,LeDonneLeonardiMontiVittone}). In the case $s=4$, there are almost no general results. The only exception is work \cite{BarilariChitourJeanPrandiSigalotti}, where the authors study smoothness of abnormal geodesics on sub-Riemannian manifolds of rank 2 and step 4.

In 2016, Le Donne and Hakavuori achieved a remarkable result in~\cite{HakavuoriLeDonne}. They proved that sub-Riemannian abnormal geodesics cannot have corners. Notably, their proof does not rely on the Pontryagin maximum principle, and instead employs an idea of using an appropriate blow-up procedure. This result has been used by several other mathematicians, including~\cite{SilvaFigalliParusinskiRifford}, where authors establish that geodesics on an analytic three-dimensional sub-Riemannian manifold must be $C^1$-smooth, regardless of the step value. The above mentioned work~\cite{BarilariChitourJeanPrandiSigalotti} also draws on their result. Nevertheless, it should be noted that the absence of corners does not necessarily imply the smoothness of abnormal geodesics on  general sub-Riemannian manifolds.

Another interesting result is presented in~\cite{MontiHolder}. The authors prove that the derivatives of geodesics satisfy the H\"older condition in the special case of left-invariant sub-Riemannian structures on Carnot groups of rank 2 that additionally satisfy the following strong restriction: $[\mathfrak{g}_i,\mathfrak{g}_j]=0$ for $i,j\ge 2$ and $i+j>4$ (here $\mathfrak{g}=\bigoplus_{j=1}^s\mathfrak{g}_j$ is the corresponding Carnot stratification of the Lie algebra $\mathfrak{g}$).

However, despite the constant efforts of leading mathematicians, smoothness of abnormal sub-Riemannian geodesics has not yet been proven. There are well-known examples \cite[section 12.6.1]{AgrachevBook} of abnormal sub-Riemannian extremals (but not geodesics) that are non-smooth. Nonetheless, to date, there are not a single example of a non-smooth abnormal geodesic. The present work provides an explanation for this gap.

\medskip

We develop an interpolational estimate on abnormal controls and apply convex duality theory to the estimate, which allows us to prove the following result: 

\begin{theorem}
	\label{thm:face}
	Any arc length parametrized geodesic on a sub-Riemannian manifold of constant rank has $L_p$-H\"older continuous derivative for any $1 \leq p < \infty$.
\end{theorem}

Therefore, in order to find a non-smooth abnormal geodesic, it is necessary to search within a narrow class of curves that have derivatives which are $L_p$-H\"older continuous with an exponent of $0<\alpha\leq 1$, and are not $C^1$-smooth. It should be noted that this class of curves is indeed very narrow, as for example, if $\alpha>1/p$, the class is in fact empty.

\medskip 

Theorem~\ref{thm:face} is formulated in a qualitative form. We state a quantitative formulation of this result in Theorem~\ref{thm:main} (see Section~\ref{sec:main_thm_statement}). Namely, Theorem~\ref{thm:main} gives exact bounds on the exponent~$\alpha$. These bounds depend only on $p$ and the step $s$ of the sub-Riemannian manifold. Some interesting corollaries of the main Theorem~\ref{thm:main} are given in Section~\ref{sec:corollaries}. They concern

\begin{itemize}
	
	\item the decay of the Fourier coefficients of controls on abnormal geodesics;
	
	\item the rate of the approximation of abnormal controls by smooth functions;
	
	\item the lower estimate of the coefficient in $L_p$-H\"older condition with an exponent $\alpha$, which is similar to the Poincar\'e inequality.
	
	\item the compactness of the set of geodesics in the space of Bessel's potentials $H^\alpha_p$.
	
\end{itemize}

\section{Thanks} Authors would like to express their deep gratitude to A.S.~Kochurov for a very helpful discussion, to A.I.~Tulenev for valuable advice in the theory of interpolation, and to L. Rizzi for a very good comment on Fourier series of optimal control. We are deeply grateful to A.A.~Agrachev and Ju.L.~Sachkov for their generous desire to share their knowledge of sub-Riemannian geometry with anyone who are able to listen (authors used this opportunity many times). We express our deep gratitude to M.~Sigalotti, D.~Prandi, A.~Hakavuori, and R.~Monti for their brilliant reports on Moscow seminar on geometric theory of optimal control, which give a priceless influence on the authors understanding of the nature of sub-Riemannian abnormal geodesics.

\section{Besov spaces and \texorpdfstring{$L_p$}{Lp}-H\"older condition.}

We start with a brief explanation on the concept of $L_p$-H\"older condition. Choose an arbitrary function $u\in L_p([t_1;t_2]\to\R^k)$. In that follows, we write $L_p(t_1;t_2) = L_p([t_1;t_2]\to\R^k)$ for short. Put $(\sigma_hu)(t)=u(t+h)$.

\begin{definition}
	
	Let $1\le p\le \infty$, $\delta>0$ be sufficiently small. The following function on $-\delta\le h\le\delta$ is called $L_p$-modulus of continuity of $u$ over $[t_1;t_2]$:
	\[
		\omega_p(h,u) =\|\sigma_h u-u\|_{L_p(I_{\mathrm{sign}\, h})}.
	\]
	where $I_+=(t_1;t_2-\delta)$ and $I_-=(t_1+\delta;t_2)$.
\end{definition}

It is evident that, $\omega_p(0,u)=0$ and $\omega_p(h,u)\to 0$ for $h \to 0$ (if $p \ne \infty$). In the case of smooth manifolds, one can introduce modulus of continuity as follows. Let $M$ be a smooth manifold and $g$ be a Riemannian metric on it. Consider an absolutely continuous curve $x:[t_1;t_2] \to M$ and put $u(t)=\dot x(t)$. In the definition of $L_p$-modulus of continuity, it is sufficient to replace $\sigma_h$ by the parallel transport along the curve $x(t)$ over time $h$ with the help of Levi-Civita connection, and to define the norm with the help of $g$. Let us remark that this definition depends on the choice of $g$ but not moderately high. Namely, let $g$ and $\tilde g$ be two Riemannian metrics on $M$. Denote by $\omega_p(\cdot,h)$ and $\tilde\omega_p(\cdot,h)$ the corresponding $L_p$-moduli of continuity. In this case, for any compact $\K\subset M$, there exists a constant $C\ge 1$ such that if $x(t)\in \K$ and $u(t)=\dot x(t)$, then

\[
	\frac1C \omega_p(u,h) \le \tilde\omega_p(u,h) \le C\,\omega_p(u,h).
\]

\begin{definition}
	\label{defn:Holder}
	
	A function $u$ is called  $L_p$-H\"older continuous on $[t_1;t_2]$ if there exist $\alpha>0$ and $c>0$ such that the following inequality (called $L_p$-H\"older condition) is fulfilled for all $0\le |h| \le \delta$:
	\[
		\omega_p(h,u)\le c |h|^\alpha.
	\]
	The number $\alpha$ is called the exponent of the $L_p$-H\"older condition. The minimal possible constant $c$ we denote by $c^\alpha_p(u)$:
	\[
		c^\alpha_p(u) = \sup_{0<|h|\le\delta} \frac{\omega_p(u,h)}{|h|^\alpha}.
	\]
	
\end{definition}

It is easy to check that the definition of $L_p$-H\"older continuity with a given exponent $\alpha$ does not depend on the choice of $\delta>0$ (but constant $c^\alpha_p(u)$ does) if $\delta$ is sufficiently small. It is evident that in the case $p=\infty$, we obtain usual definitions of the modulus of continuity and H\"older continuous functions. The moduli of continuity are directly related to the Besov and Sobolev spaces (see~\cite[5.3]{Stein}), which will be repeatedly used in this paper. 

\begin{definition}
	\label{defn:Besov_space}
	The Nikolski--Besov (or simply Besov\footnote{According to the authors knowledge, spaces $B^{\alpha}_{p,\infty}$ were first introduced and studied by Nikolski. Then his student Besov introduced spaces $B^{\alpha}_{p,q}$ for arbitrary $1\le q\le\infty$ and developed the corresponding theory. Then spaces $B^{\alpha}_{p,q}$ become much more convenient to work with than just the Nikolski spaces $B^{\alpha}_{p,\infty}$ because of the very nice interpolational properties of the Besov spaces $B^{\alpha}_{p,q}$ with $q\ne 1,\infty$.}) space $B^{\alpha}_{p,\infty}(t_1,t_2)$ (where $0<\alpha<1$) consists of functions $u\in L_p(t_1,t_2)$ that satisfy
	\[
		\|u\|_{B^{\alpha}_{p,\infty}(t_1,t_2)} \eqdef \|u\|_p + c^\alpha_p(u) < \infty.
	\]
\end{definition}

Let us remark that the different choices of $\delta>0$ produces different norms on $B^{\alpha}_{p,\infty}(t_1,t_2)$, but all of them are pairwise equivalent if $\delta$ is sufficiently small (see \cite[4.4.1]{Triebel}). Thus, the Nikolski--Besov space $B^{\alpha}_{p,\infty}(t_1,t_2)$ consists of functions $L_p(t_1,t_2)$ that are $L_p$-H\"older continuous with the exponent $\alpha\in(0;1)$.

In what follows, we use interpolation theory of Banach spaces and theory of Besov spaces very often. For the lack of the possibility to give a short introduction into so solid part of mathematics, we refer the reader to the excellent book \cite{Triebel} for details.

\section{The main result}
\label{sec:main_thm_statement}

Let $(M,\Delta,g)$ be a smooth sub-Riemannian manifold of constant rank, i.e.~$M$ is a smooth manifold, $\dim M=n$, $\Delta(x)\subset T_xM$ is a smooth distribution with a fixed dimension $\dim \Delta(x)\equiv k<n$, and $g(x)$ is a dot product on $\Delta(x)$, which smoothly depends on $x$. Besides, following~\cite{AgrachevBook}, we include into the definition the demand that the distribution $\Delta(x)$ meets the H\"ormander condition (the bracket-generating condition). A Lipschitz continuous curve $x(t)\in M$, $t\in[0;1]$, is called horizontal (or admissible) if $\dot x(t)\in \Delta(x(t))$ for a.e.~$t$. The length of a (horizontal) curve is defined in the standard way:
\[
l(x) = \int_0^1 |\dot x(t)|_{g(x(t))}\,dt.
\]

Let us remind shortly various definitions of local minima (see \cite{AgrachevSmoothMinimum}). Let $x(t)$ be a horizontal curve.

\begin{itemize}
	
	\item[$(G)$] Curve $x(t)$ is called sub-Riemannian shortest path if any other horizontal curve $y(t)$ that has $x(0)=y(0)$ and $x(1)=y(1)$ meets the condition $l(x)\le l(y)$.
	
	\item[$(C^0)$] Curve $x(t)$ is called $C^0$-local minimum if there exists a neighborhood $U\subset M$ of the curve $x(t)$ such that any horizontal curve $y(t)$ that lies in $U$ and has  $x(0)=y(0)$, $x(1)=y(1)$ meets $l(x)\le l(y)$.
	
	\item[$(W^1_p)$] Curve $x(t)$ is called $W^1_p$-local minimum for
	$p=1$ or $p=\infty$  if there exists $\varepsilon>0$  such that any horizontal curve $y(t)$ that satisfies $x(0)=y(0)$, $x(1)=y(1)$, and\footnote{Here we use any finite trivialization $TM$ in a neighborhood of curve $x(t)$. The choice of trivialization can affect on $\varepsilon$, but it do not affect on the definition of the local minimum itself.} $\|\dot x-\dot y\|_p<\varepsilon$ meets $l(x)\le l(y)$.
\end{itemize}

Sometimes $C^0$-local minimum is called strong, $W^1_\infty$ --- weak, and $W^1_1$ --- Pontryagin's. It is evident that there are the following implications: $G\Rightarrow C^0\Rightarrow W^1_1\Rightarrow W^1_\infty$ (see\cite{AgrachevBook}).

\begin{theorem}
	\label{thm:main}
	
	Let $\K\subset M$ be a compact set and distribution $\Delta$ at points of $\K$ meet the H\"ormander condition of step that is not greater than $s$.
	
	\begin{itemize}
		\item[\cw] If $x(t)$ is a $W^1_1$-local minimum of the sub-Riemannian distance, $x(t)\in \K$, and $|\dot x(t)|_g\equiv\const$, then derivative $\dot x$ is $L_2$-H\"older continuous with exponent $\alpha=\frac{1}{s-1}$, i.e.
		\[
			\dot x\in B^\frac1{s-1}_{2,\infty}(0;1).
		\]
		Moreover, for any $2\le p<\infty$, the derivative $\dot x$ is $L_p$-H\"older continuous with any exponent $\alpha<\frac{2}{p(s-1)}$, i.e.~$\dot x\in B^\alpha_{p,\infty}(0;1)$.
		
		\item[\cg] For any $2+\frac{1}{s-1}\le p<\infty$, $0<\beta<\frac{1}{p(s-1)}$, and $0<\kappa<1-\beta(s-2)$, there exists a constant\footnote{Constant $C$ depends on $\K$, $p$, $\beta$, and $\kappa$, but it does not depend on $x(\cdot)$.} $C>0$ such that any sub-Riemannian shortest path $x(t)\in\K$ with $|\dot x(t)|_g\equiv l(x)$ satisfies
		\[
			\|\dot x\|_{B^\beta_{p,\infty}(0;1)}\le C l(x)^\kappa.
		\]

	\end{itemize}
\end{theorem}

The full proof is given in Section \ref{sec:proof}. It is based on a new approach on how to investigate smoothness of sub-Riemannian optimal controls. We call it \textbf{the duality method of interpolational estimates}. Namely, we intend to construct an upper estimate on any optimal control $u$ of the following form:
\[
	\exists q,r\ge 1\ \exists\theta\in(0;1)\ \exists c>0
	\ \ :\ \ 
	\forall \varphi\in C_0^\infty(0;1)\qquad
	\int_0^1 u\dot\varphi\,dt \le c\|\varphi\|_r^{1-\theta}\|\dot\varphi\|_q^\theta.
\]
This estimate appears to be dual to an estimate on $K$-functional of an interpolation of the corresponding Sobolev spaces\footnote{Detailed definitions and properties of $K$-functional in general theory of interpolation can be found in \cite{Stein, Triebel}. The $K$-functional corresponding to our case is defined in~\eqref{eq:K_functional}.}. This leads to a possibility to prove that optimal control $u$ does satisfy $L_p$-H\"older condition.

\section{Corollaries of the main theorem on \texorpdfstring{$L_p$}{Lp}-H\"older continuity of sub-Riemannian geodesics derivatives.}

\label{sec:corollaries}

Theorem~\ref{thm:main} has many convenient consequencies. For instance, Theorem \ref{thm:face} that is formulated in the introduction. Here we give its proof.

\begin{proof}[Proof of Theorem~\ref{thm:face}]
	
	The H\"ormander condition guarantees that $s<\infty$ on compact $\K$ since the step is a lower semi-continuous function and, hence, it reaches its minimum on compact $\K$. Further, a geodesic can be divided into a finite number of parts, and it is a shortest path on each of them. Hence, according to Theorem~\ref{thm:main}, the velocity on each of these parts satisfies $L_p$-H\"older condition for any $2\le p<\infty$. It remains to note that $L_p$-H\"older condition for some $p$ implies $L_q$-H\"older condition for any $1\le q\le p$.
	
\end{proof}

Now, let us suppose  for the sake of statements simplicity that
\[
	\Delta(x)=\mathrm{span}(f_1(x),\ldots,f_k(x)).
\]
In this case, horizontal curves have velocities of the form   
\[
	\dot x(t) = \sum_{j=1}^k u_j(t) f_j(x(t)),
\]
where $u=(u_1,\ldots,u_k)$ is a control.

Since $x(t)\in W^1_\infty(0;1)$, one obtains  $f_j(x(t))\in W^1_\infty(0;1)$ and $g(x(t))\in W^1_{\infty}(0;1)$. So control           
\[
	u(t) = g(x(t))[f_i(x(t)),f_j(x(t))]^{-1}g(x(t))[\dot x(t),f_j(x(t))]
\]
appears to satisfy $L_p$-H\"older condition with the same exponent $\alpha$ as $\dot x(t)$. Moreover, the part \cg of Theorem~\ref{thm:main} remains true with $u$ in place of $\dot x$ (only constant $C$ may differ) as fields $f_j$ and their derivatives $\frac{\partial}{\partial x} f_j$ are bounded on compact $\K$.

\medskip

The fact that control $u$ is $L_p$-H\"older continuous makes it possible to estimate rate of decay for Fourier series coefficients. Let $\ell_\gamma=\ell_\gamma(\R^k)$ for $\gamma\ge 1$ denote the space of sequences $\{v^m\}_{m=-\infty}^\infty$, $v^m\in\R^k$, such that $\sum_m|v^m|^\gamma<\infty$.

\begin{corollary}[on the Fourier coefficients]

	Let $\hat u^m$, $m\in\Z$, be the Fourier coefficients of the function $u\in L_\infty(0;1)$ and $S_N$ be the partial Fourier sum, $S_N=\sum_{m=-N}^N\hat u^me^{2\pi\I mt}$. 
	
	\begin{itemize}
		
		\item[\cw] Let $x(t)$ be either a geodesic or $W^1_1$-local minimum and $|u(t)|_g\equiv l(x)$ be the corresponding control. Then for any\footnote{Here $s\ge2$ is the maximal step at points of the curve $x$.} $\alpha\le\min\{\frac{1}{s-1},\frac12\}$, $\alpha\ne\frac{1}{s-1}$ and $\gamma>\max\{1,2-\frac{4}{s+1}\}$ one has\footnote{Authors would like to express their gratitude to professor L. Rizzi for a very useful discussion about the case $\alpha>\frac12$.}	
		\[
			\sum_{m\in\Z} |m|^{2\alpha}|\hat u^m|^2 < \infty;
			\qquad
			\hat u^m \in \ell_\gamma;
			\quad\text{and}\quad
			\exists C>0:\ \forall N\ \ \|u-S_N\|_2 \le CN^{-\alpha}.
		\]
		
		\item[\cg] If $\K\subset M$ is compact, then for any $\beta<\frac1{2s-1}$, $0<\kappa<1-\beta(s-2)$, and $\delta>\frac{2}{1+2\beta}$, there exists a constant $C$ such that for any shortest path $x$ lying in $\K$ with the control $|u(t)|_g\equiv l(x)$, one has	
		\[
			\sum_{m\in\Z} |m|^{2\beta}|\hat u^m|^2 < Cl(x)^\kappa;
			\qquad
			\|\hat u^m\|_{\ell_\delta} \le Cl(x)^\kappa;
			\quad\text{and}\quad
			\forall N\ \ \|u-S_N\|_2 \le Cl(x)^\kappa N^{-\beta}.
		\]
		
	\end{itemize}
	
\end{corollary}

\begin{proof}
	
	These are well known properties of the Fourier coefficients of functions being $L_p$-H\"older continuous. Let us prove, say, \cw (the proof of \cg is fully similar). Control $u$ satisfies $L_2$-H\"older condition with any index $0<\alpha<\frac{1}{s-1}$ due to Theorem \ref{thm:main} (in the case \cw, one can always choose, say, the graph of $x$ as $\K$). We start with proving the third assertion. The main tool here is \cite[Theorem 1.3.9]{Pinsky}. Pinsky uses a slightly different shifting operator $\hat\sigma_h$ (instead of $\sigma_h$ used in the present paper) to define $L_p$-H\"older continuity. Namely, he defines the function $\hat\sigma_hu$ to be the periodic shift of $u$:
	\[
		\hat\sigma_hu(t) = u(\{t+h\}\text)\text{ for a.e.\ }t\in[0;1]
	\]
	(here $\{\cdot\}$ stand for the fractional part of a number). We will say that $u$ is periodically $L_p$-H\"older continuous with exponent $\hat\alpha$, if it is $L_p$-H\"older continuous wrt periodic shift $\hat\sigma_h$, i.e.\ if there exists $c>0$ such that
	\[
		\hat\omega_p(h,u)\eqdef\|u-\hat\sigma_hu\|_{L_p(0;1)}\le \hat c|h|^{\hat \alpha}.
	\]

	It is easy to see that if $u\in L_\infty(0;1)$, then these two definitions of $L_p$-H\"older continuity  are equivalent, but the exponents may differ. Indeed, on the one hand, $\omega_p(h,u)\le \hat\omega_p(h,u)$. So, if $u$ is periodically $L_p$-H\"older continuous with an exponent $\hat\alpha$, then it is $L_p$-H\"older continuous with the same exponent $\hat\alpha$ (or greater). On the other hand, for $h>0$, we have
	\begin{multline*}
		\hat\omega_p(-h,u)^p = \hat\omega(h,u)^p = 
			\int_0^{1-h}|\sigma_hu-u|^p\,dt + \int_{1-h}^1 |\sigma_{-1+h}u-u|^p\,dt =\\
		 =\omega_p(h,u)^p + (\|\sigma_{-1+h}u-u\|_{L_p(1-h,1)})^p \le 
		 	\omega_p(h,u)^p + (2\|u\|_\infty h^{1/p})^p = 
		 	\omega_p(h,u)^p + 2^p\|u\|_\infty^p h
	\end{multline*}
	So if $u\in L_\infty(0;1)$ is $L_p$-H\"older continuous with an exponent $\alpha$, then it is periodically $L_p$-H\"older continuous with the exponent $\min\{\alpha,\frac1p\}$ (or greater). Summarizing,
	\[
		\min\big\{\alpha,\frac1p\big\}\le\hat\alpha\le\alpha.
	\]

	\cite[Theorem 1.3.9]{Pinsky} states that if $u$ is periodically $L_2$-H\"older continuous with an exponent $\hat\alpha$, then $\|u-S_N\|_2 \le CN^{-\hat\alpha}$ and the third assertion is proved. The first one follows from the third by using \cite[Exercise 1.3.15]{Pinsky}. The second assertion can be found, for instance, in \cite{Neugebauer} by using the inclusion $B^\alpha_{p,\infty}(0;1)\subset B^{\alpha-\varepsilon}_{p,q}(0;1)$, which is valid for all $\varepsilon>0$ and $1\le q\le\infty$ \cite[4.6.1]{Triebel}.
	
	The case \cg has a very similar proof. The only differences are the following: (i) since $u\in B_{2+\frac{1}{s-1},\infty}^{\beta}(0;1)$ for any $\beta<\frac{1}{2s-1}$, we have $u\in B_{2,\infty}^{\beta}(0;1)$; (ii) since $s\ge 2$, we have $\min\{\frac1{2s-1},\frac12\}=\frac1{2s-1}$.
	
\end{proof}

Another property of $L_p$-H\"older continuity of the control is a possibility to approximate it well by smooth functions. This property is non-trivial. There are a lot of functions in $L_\infty(0;1)$ that can not be approximated well by smooth functions. Namely, there exists $u\in L_\infty(0;1)$ such that for any series of functions $w_\varepsilon\in C^1$ approximating $u$, $\|u-w_\varepsilon\|_2<\varepsilon$, and any $\gamma>0$
one has $\varepsilon^\gamma\|\dot w_\varepsilon\|_\infty\to \infty$ as $\varepsilon\to +0$. In other words, while one is approximating $u$ by smooth functions, their derivatives grow faster than any negative power of $\varepsilon$. However, for functions that are $L_p$-H\"older continuous, it is possible to select smooth approximating functions that have much smaller derivatives. In the corollary below, we suggest one possible way of choosing norms, but it is possible to use other norms with the help of interpolation methods (see \cite{Triebel}).

\begin{corollary}[on the approximation rate by smooth functions]
	Let $x(t)$  be either a geodesic or $W^1_1$-local minimum and $|u(t)|_g\equiv l(x)$ be a control. If\footnote{Where $s$ is the sub-Riemannian step as always.} $\gamma > s-1$, then there exists a constant $C>0$ such that, for any $\varepsilon>0$, there is a function $w\in C^\infty(\R)$ such that
	\[
		\|w-u\|_{L_2(0;1)} \le C\varepsilon
		\quad\text{and}\quad
		\varepsilon^\gamma\|\dot w\|_{L_\infty(0;1)}\le C.
	\]
\end{corollary}

\begin{proof}
	This approximation is obtained by the following chain of inclusions. We first fix $0<\alpha<1$. Then for $\delta>0$ and $2\le r<\infty$, we have
	\[
		(L_2(0;1),W^1_\infty(0;1))_{\alpha,\infty} \supset
		(B^\delta_{2,2}(0;1),B^{1+\frac1r}_{r,r}(0;1))_{\tilde\alpha,p}=
		B^{(1-\tilde\alpha)\delta + \tilde\alpha(1+\frac1r)}_{p,p}(0;1)\supset
		B^{(2-\tilde\alpha)\delta + \tilde\alpha(1+\frac1r)}_{p,\infty}(0;1).
	\]
	The first inclusion is fulfilled for $\alpha<\tilde\alpha$ and any $1\le p\le\infty$ (see \cite[1.3.3]{Triebel}). The second equality is fulfilled for $\frac1p=\frac{1-\tilde\alpha}2+\frac{\tilde\alpha}{r}$ according to \cite[4.3.1]{Triebel} (and so $p\ge 2$ automatically) --- this is so called diagonal case of interpolation of Besov spaces. The last embedding is of Sobolev type~\cite[4.6.2]{Triebel} since $\delta>0$. Now we apply the main Theorem~\ref{thm:main}: the control $u$ belongs to the last space if
	\[
		(2-\tilde\alpha)\delta + \tilde\alpha\left(1+\frac1r\right)<\frac2{p(s-1)}=
		\frac{2}{s-1}\left(\frac{1-\tilde\alpha}2+\frac{\tilde\alpha}{r}\right)
	\]
	or, equivalently,    	
	\[
		\tilde\alpha\left(
		1+\frac{1}{s-1}+\frac1r-\delta +\frac{2}{r(s-1)}
		\right) <
		\frac{1}{s-1}-2\delta.
	\]
	For $\delta\to +0$ and $r\to\infty$, the factor in the left-hand side tends to $1+\frac{1}{s-1}$ and the right-hand side tends to $\frac{1}{s-1}$. Therefore, this inequality is compatible with $\tilde\alpha>\alpha$ if and only if $\alpha(1+\frac1{s-1})<\frac1{s-1}$ or $\alpha<\frac1s$. 
	
	Under this choice of $\alpha$, one obtains  $u\in (L_2(0;1),W^1_\infty(0;1))_{\alpha,\infty}$. It means that the $K$-functional of the interpolation~\cite[1.3]{Triebel} satisfies the following estimate:
	\[
		\sup_{M>0}(M^{-\alpha}K(M,u)) = 
		\sup_{M>0}\inf_{w\in W^1_\infty(0;1)} (M^{-\alpha}\|w-u\|_2 + M^{1-\alpha}\|w\|_{W^1_\infty})=C<\infty.
	\]
	Space $C^\infty(\R)\cap W^1_p(0;1)$ is dense in $W^1_p(0;1)$ (see\cite[2.3.2]{Triebel}). Let us denote $M^\alpha=\varepsilon$. Then for any $\varepsilon>0$, there exists a function $w\in C^\infty(\R)$ such that $\|w-u\|_{L_2(0;1)} \le 2C \varepsilon$ and $\|\dot w\|_{L_\infty(0;1)}\le 2C/\varepsilon^{(1-\alpha)/\alpha}$.  
\end{proof}

Theorem \ref{thm:main} also implies an estimate on constant $c_p^\alpha(u)$ in $L_p$-H\"older condition (see Definition~\ref{defn:Holder}), which is similar to the Poincar\'e inequality

\begin{corollary}[an analog of the Poincar\'e inequality]
	\label{cor:Poincare}
	For any $p$ and $\alpha$ from item \cw of Theorem \ref{thm:main}, there exists a constant $C>0$ such that for any geodesic (or $W^1_1$-local minimum) $x(t)$ lying in $\K$ with control $|u(t)|_g\equiv l(x)$, one has
	\[
		c^\alpha_p(u) \ge C\|u-\bar u\|_p
	\]
	where $\bar u=\int_0^1u(t)\,dt\in\R^k$ is the mean value of the control $u(t)$ on the interval $t\in[0;1]$.
\end{corollary}

In particular, if the system of fields $f_j$ is orthonormal, then for $p=2$ and any $0<\alpha<\frac{1}{s-1}$, there exists $C>0$ such that
\[
	c^\alpha_2(u) \ge C\sqrt{l(x)^2 -\bar u^2}
\]
since, in this case,
\[
	\|u-\bar u\|_2^2 = \int_0^1 (u(t)-\bar u)^2\,dt = 
	\int_0^1 (u^2(t)-2\bar uu(t) + \bar u^2)\,dt =
	l(x)^2-\bar u^2.
\]
The proof of Corollary \ref{cor:Poincare} is similar to the proof of the classic Poincar\'e inequality. Let us give it shortly.

\begin{proof}[Proof of Corollary \ref{cor:Poincare}]
	
	Let $Y\subset B^\alpha_{p,\infty}(0;1)$ consists of (not necessarily optimal) controls with zero average: $Y=\{u:\bar u=0\}$. First, we prove the inequality for $u\in Y$. Suppose the opposite: for any constant $C>0$ there exists $u\in Y$ such that $c^\alpha_p(u)\le C\|u\|_p$. In this case, there exists a sequence $u^m\in Y$ such that $c^\alpha_p(u^m)/\|u^m\|_p\to0$ as $m\to\infty$. Without loss of generality, we assume that $c^{\alpha}_p(u^m)+\|u^m\|_p=1$ for all $m$. 
	
	The space $B^\alpha_{p,\infty}(0;1)$ can be represented as a dual space
	$B^\alpha_{p,\infty}(0;1)=(\tilde B^{-\alpha}_{p^*,1}(0;1))'$ where $\frac1p+\frac1{p^*}=1$ (see \cite[4.8.1]{Triebel}). Therefore, due to Banach--Alaoglu theorem, it is possible to choose a weak$^*$-convergent subsequence of the sequence $u^m$. We keep indices $m$ for the subsequence without loss of generality. Space $B^\alpha_{p,\infty}(0;1)$ is compactly embedded into the space $L_p(0;1)$ according to the Rellich-Kondrachev theorem (see the proof of Theorem~\ref{thm:compact} for details). Hence, there is a subsequence that converges strongly in $L_p(0;1)$. We also keep indices $m$ for it.
	
	Thus, we obtain a sequence $u^m$ that converges to $u$ strongly in $L_p$ and weakly$^*$ in $B^\alpha_{p,\infty}(0;1)$ as $m\to\infty$. It is evident that $u\in Y$.
	
	Recall that $c^\alpha_p(u^m)/\|u^m\|_p\to0$ as $m\to\infty$. Therefore, $\|u^m\|_p\le 1$ implies $c^\alpha_p(u^m)\to 0$. So $\|u^m\|_p\to 1$ and $\|u\|_p=1$. Moreover, $u$ is the weak$^*$-limit of $u^m$, which implies
	\[
		\|u\|_p + c^\alpha_p(u) = \|u\|_{B^\alpha_{p,\infty}}\le\limsup_{m\to\infty} \|u^m\|_{B^\alpha_{p,\infty}} = 1.
	\]
	So $c^\alpha_p(u)=0$ and the function $u$ is constant, which is impossible since $\bar u=0$ and $\|u\|_p=1$. Thus, the statement is proved for any $u\in Y$ by contradiction.
	
	For an arbitrary function $u\in B^\alpha_{p,\infty}(0;1)$, one has $u-\bar u\in Y$ and $c^\alpha_p(u-\bar u) = c^\alpha_p(u)$.	It remains to note that $x(t)$ is a geodesic or $W^1_1$-local minimum. Hence, its control belongs to space $B^\alpha_{p,\infty}(0;1)$ by Theorem~\ref{thm:main}.	
	
\end{proof}

Theorem \ref{thm:main} also gives us a possibility to prove a reinforced version of the Agrachev result on compactness of the set of shortest paths \cite{AgrachevCompact}. Denote the set of all shortest paths lying in $\K$ by
\[
	\mathfrak{X}\subset\left\{
	(x^0,x^1,u)\in \K\times \K\times L_\infty(0;1)
	\right\}.
\]
Here we construct triplets $(x^0,x^1,u)$ in the following way. We take an arbitrary shortest path $x(t)\in W^1_\infty(0;1)$ lying in $\K$ that has a constant velocity $|\dot x(t)|_g\equiv\const$. Then $x^0=x(0)$ and $x^1=x(1)$ are the initial and final points of $x(t)$, and $u(t)$ is its control. So $\mathfrak{X}$ consists of all triplets $(x^0,x^1,u)$ corresponding to all shortest paths $x(\cdot)$ in $\K$.

\begin{theorem}[on compactness of the set of shortest paths in the space of the Bessel potentials]
	\label{thm:compact}
	
	Let $\K\subset M$ be a compact set and $f_j$ form an orthonormal frame.	Then $\mathfrak{X}$ is a compact subset in $\K\times \K\times H^\beta_p(0;1)$ where $H^\beta_p(0;1)$ is the space of Bessel potentials\footnote{The assertion of the theorem remains true if one exchange the set $H^\beta_p(0;1)$ by the Besov space $B^\beta_{p,q}(0;1)$ $1\le q\le\infty$ or by the Sobolev--Slobodeckij space $W^\beta_p(0;1)$ due to the continuity of the Sobolev embeddings $H^{\beta+\varepsilon}_p(0;1)\subset B^\beta_{p,q}(0;1)$ and $H^{\beta+\varepsilon}_p(0;1)\subset W^\beta_p(0;1)$ for any $\varepsilon>0$ (see \cite[4.6.2]{Triebel}).} and $p$ and $\beta$ are from item \cg of Theorem	\ref{thm:main} or $1\le p<2+\frac1{s-1}$, $\beta<\frac1{2s-1}$. 
	
\end{theorem}

\begin{proof}
	
	Let us note	that one can assume that fields $f_j$ are complete without loss of generality. Indeed, in the opposite case, we proceed as follows. Denote by $V$ an open neighborhood of $\K$ such that its closure is compact. Take a smooth function $\psi$ such that $\psi(x)=1$ for $x\in\K$, $\psi(x)=0$ for $x\not\in V$ and $0\le\psi(x)\le 1$ for all $x$. We exchange fields $f_j$ by that of $\tilde f_j=\psi f_j$ and exchange $g$ appropriately. Under such a change (i) all fields become complete; (ii) the rank becomes nonconstant, but remains constant in a neighborhood of $\K$ (which is obviously enough for Theorem~\ref{thm:main}); (iii) horizontal curves in $\K$ remain the horizontal ones; (iv) a part of horizontal curves lying out of $\K$ cease to be horizontal; and (v) no new horizontal curves arise. Moreover, lengths of horizontal curves in $\K$ remain unchanged, as for the rest, their lengths may only increase. The estimate in item \cg of Theorem~\ref{thm:main} does not altered since the shortest curves in $\K$ do not altered.	
	
	Without loss of generality, assume that $p$ and $\beta$ are from from item \cg of Theorem~\ref{thm:main}. Indeed, the statement for $p<2+\frac1{s-1}$ and $\beta<\frac1{2s-1}$ follows from the continuity of the inclusion $B^\beta_{2+\frac1{s-1},\infty}(0;1)\subset B^\beta_{p,\infty}(0;1)$.
	
	At first, let us show that the set of all controls on shortest paths lying in $\K$ belongs to $H^\beta_p(0;1)$ and forms a precompact set in it. The embedding of $W^1_p(0;1)$ into $L_p(0;1)$ is compact by the Rellich--Kondrachev theorem \cite[5.7]{Evance}. Take some numbers $\beta_{j}$, $j=1,2$, such that $\beta<\beta_1<\beta_2<\frac{1}{p(s-1)}$. Then the Sobolev--Slobodeckij spaces $W^{\beta_j}_p(0;1)=B^{\beta_j}_{p,p}(0;1)$ (see \cite[2.4.1]{Triebel}) can be represented as the following interpolational ones (see \cite[4.3.1]{Triebel}):
	\[
		W^{\beta_j}_p(0;1) = (L_p(0;1),W^1_p(0;1))_{\beta_j,\,p}.
	\]
	Since $W^1_p(0;1)\Subset L_p(0;1)$, we have $W^{\beta_2}_p(0;1)\Subset W^{\beta_1}_p(0;1)$ (see~\cite[1.16.4]{Triebel}). The set of controls on shortest paths lying in $\K$ is bounded in $W^{\beta_2}_p(0;1)$ by Theorem~\ref{thm:main} (since $B^{\beta_2+\varepsilon}_{p,\infty}(0;1)\subset B^{\beta_2}_{p,p}(0;1)$ for any $\varepsilon>0$, \cite[4.6.2]{Triebel}) and thus it is precompact in~$W^{\beta_1}_p(0;1)$. Now we recall that $W^{\beta_1}_p(0;1)\subset H^\beta_p(0;1)$ by the corresponding  embedding of Sobolev type (see \cite[4.6.2]{Triebel}). Thus, the set of controls on shortest paths in $\K$ is precompact in $H^\beta_p(0;1)$.
	
	So, the set $\mathfrak{X}$ is precompact in $\K\times \K\times H^\beta_p(0;1)$. In order to show that it is compact, let us consider two maps
	\[
		\begin{gathered}
			e,i:\K\times \K\times L_p(0;1)\to \K\times \K\times \R,\\
			e(x^0,x^1,u) = (\tilde x^0,\tilde x^1,\|u(t)\|_p),\\
			i(x^0,x^1,u) = (x^0,x^1,d_{SR}(x^0,x^1)).
		\end{gathered}
	\]
	Here $d_{SR}(x^0,x^1)$ is the sub-Riemannian distant between $x^0$ and $x^1$; $\tilde x^1$ is the value of the unique solution to the differential equation $\dot x=\sum_j u_jf_j(x)$ with the initial condition $x(0)=x^0$ at $t=1$; and $\tilde x^0$ is the value of the unique solution to the same equation, but with the initial condition $x(1)=x^1$ and at~$t=0$.
	
	We claim that since $p>1$ (in fact, $p>2+\frac{1}{s-1}$), we have
	\[
		\mathfrak{X} = \big\{(x^0,x^1,u)\ \Big|\ e(x^0,x^1,u)=i(x^0,x^1,u)\Big\}.
	\]
	Indeed, inclusion $\subset$ is evident and inclusion $\supset$ is also not so difficult to obtain. If the first two components of $e$ and $i$ coincide, then control $u$ leads from $x^0$ to $x^1$. If also the last components of $e$ and $i$ coincide, then length of any horizontal curve $x(t)$ having $u(t)$ as a control (that is $\|u(t)\|_1$) is not greater than $d_{SR}(x^0,x^1)$ (since $\|u(t)\|_1\le \|u(t)\|_p=d_{SR}(x^0,x^1)$). In addition, $d_{SR}(x^0,x^1)\le \|u(t)\|_1$ according to the definition of the sub-Riemannian distance. So the equality can happen only if $|u(t)|\equiv d_{SR}(x^0,x^1)$ for a.e.~$t$ as $p>1$.	
	
	Map $i$ is defined for all $(x^0,x^1,u)$ and it is continuous with respect to any topology on $L_p(0;1)$, and map $e$ is defined for all $(x^0,x^1,u)$ since fields $f_j$ are complete and $e$ is continuous with respect to the weak topology on $L_p(0;1)$. Therefore, set $\mathfrak{X}$ (as the set of points where $e=i$) is weakly closed in the $\K\times \K\times L_p(0;1)$, and so it is strongly closed. Hence, set $\mathfrak{X}$ is closed in $\K\times\K\times H^\beta_p(0;1)$ due to continuity of the embedding $H^\beta_p(0;1)\subset L_p(0;1)$ of Sobolev type.

\end{proof}

The previous theorem is formulated in terms of controls. But it can be easily reformulated it terms of the shortest paths.

\begin{corollary}
	
	Let $\K\subset M$ be a compact set. Then the set of shortest paths $x(t)$ in $\K$ with $|\dot x|_g\equiv l(x)$ is compact in $H^{1+\beta}_p(0;1)$ where $p$ and $\beta$ are from item \cg of Theorem	\ref{thm:main} or $1\le p<2+\frac1{s-1}$ and $\beta<\frac1{2s-1}$.
	
\end{corollary}

\begin{proof}
	
	Indeed, consider the following map 
	\[
		\mathfrak{X} \to W^1_\infty(0;1); \qquad (x^0,x^1,u)\mapsto x(\cdot)
	\]
	where $x(t)$ is the unique solution to the equation $\dot x=\sum u_jf_j(x)$ with the initial condition  $x(0)=x^0$ (we simply forget about point $x^1$). This map is surjective on the set of shortest paths in $\K$ by the definition of $\mathfrak{X}$. Since $\mathfrak{X}\subset H^\beta_p(0;1)$ by Theorem~\ref{thm:compact}, the image of $\mathfrak{X}$ belongs to $H^{1+\beta}_p(0;1)$. Moreover, the map is continuous as a map from $\K\times \K\times H^\beta_p(0;1)$ to $H^{1+\beta}_p(0;1)$. It remain to note that the continuous image of any compact set (e.g.\ $\mathfrak{X}$) is also compact.
	
\end{proof}

\section{Proof of Theorem~\ref{thm:main}}
\label{sec:proof}

Proofs of the both items \cw and \cg in Theorem~\ref{thm:main} are very close, so we write a united proof for both of them and make special notes in places where they are different. For convenience of writing the united proof, in the case \cw, we denote by $L$ the given curve length, $L=l(x)$; and in the case \cg, lengths of shortest paths in $\K$ may differ, but all of them are bounded by a constant $L$ (i.e.\  $l(x)\le L$) since $\K$ is compact and sub-Riemannian distance is continuous. In other words, $L$ is a constant such that $l(x)\le L$ in both cases \cw and \cg.

\bigskip


\textit{Outline of the proof:} The proof of Theorem~\ref{thm:main} is divided into five main parts. Firstly, in section~\ref{subsec:reduction}, we present a precise reduction of the subriemannian problem to an optimal control one. This reduction is technical and somewhat standard. Readers interested solely in the optimal control aspect may skip this section. 

Secondly, in section~\ref{subsec:simplifications}, we introduce additional assumptions on the optimal control that can be assumed without loss of generality. These assumptions simplify greatly the rest of the proof.

Thirdly, in section~\ref{subsec:interpolation}, we present the heart of the proof: a variation on an optimal trajectory that allows us to establish a key interpolational estimate on the optimal control.

Fourthly, in section~\ref{subsec:duality}, we demonstrate that the key interpolation estimate is dual to the problem of approximating optimal controls by smooth functions.

Finally, in section~\ref{sec:interpolation}, we use the approximation to obtain estimates on the $K$-functional, which leads to an interpolation of Besov spaces.

\subsection{Reduction to an optimal control problem}
\label{subsec:reduction}

Before we start the proof of Theorem~\ref{thm:main}, let us reduce it to the simplest possible form. We begin with reducing the general problem to the optimal control one. So we want to prove that, without loss of generality, one can assume that there exist vector fields $f_1,\ldots f_k\in \mathrm{Vect}(M)$ such that

\medskip

\begin{equation}
\label{oc:f_i}
\tag{$OC_1$}
	 \Delta = \mathrm{span}(f_1,\ldots,f_k);
\end{equation}

\begin{equation}
\label{oc:f_i_const}
\tag{$OC_2$}
	\begin{gathered}
	M\subset\R^n\textit{ and for each }i\textit{ there exist some global coordinates on }M\\
	\textit{ such that } f_i(x)\textit{ is constant in these coordinates};
	\end{gathered}
\end{equation}

\begin{equation}
\label{oc:f_i_orthonormal}
\tag{$OC_3$}
	g[f_i,f_j]=\delta_{ij}.
\end{equation}

\begin{proof}[Proof of~\eqref{oc:f_i}, \eqref{oc:f_i_const} and~\eqref{oc:f_i_orthonormal}.] 
	
	Obviously, any $x\in M$ has a neighborhood where vector fields satisfying~\eqref{oc:f_i} can be easily found. Moreover, vector fields $f_i$ are linearly independent since there are exactly $k=\dim\Delta$ of them. In particular, $f_i\ne 0$. Hence, each point $x\in M$ has a neighborhood diffeomorphic to a ball in $\R^n$ such that each vector field $f_i$ is constant in some local coordinates on that neighborhood. Moreover, since $f_i$ are linearly independent, we may assume that $f_i$ are orthonormal by using the standard Gram--Schmidt process.
	
	Thus, any point $x\in M$ has a neighborhood $U^x$ (which is diffeomorphic to a ball in $\R^n$) and defined on $U^x$ vector fields $f_i$ such that~\eqref{oc:f_i}, \eqref{oc:f_i_const}, and~\eqref{oc:f_i_orthonormal} are fulfilled in $\tilde M=U^x$. Let us now show that it is sufficient to prove the theorem assuming $\tilde M=U^x$.
	
	Denote $B=\{\xi\in\R^n:|\xi|<1\}$ and let $\phi^x:B\to U^x$ be the corresponding diffeomorphism with $\phi^x(0)=x$. Let us cover compact $\K$ by a finite number of neighborhoods of the form $\phi^{x_m}(\frac13B)$. Put $\phi^m=\phi^{x_m}$ for short. Further, using this atlas, for any horizontal curve $x(t)$ lying in the union of these neighborhoods, it is possible to construct in standard manner a partition of $[0;1]$ of the from $0=t_0<t_1<\ldots<t_N=1$ in such a way that, for any $0\le j< N$, there exists a (minimal) index $m_j$ such that $x(t_j)\in\phi^{m_j}(\frac13B)$; $x(t)\in \phi^{m_j}(\frac23B)$ for $t\in[t_j;t_{j+1})$; and if $j+1\ne N$, then $x(t_{j+1})\in\phi^{m_j}(\frac23\partial B)$. 
	
	Let us now using this partition construct a number of overlapping intervals that cover $[0;1]$. Denote $s^+_j=\max\{t\in[t_j;t_{j+1}]:x(t)\in \phi^{m_j}(\frac12\partial B)\}$ for $j<N-1$. Note that the motion time from a given point on $\phi^{m}(\frac13\partial B)$ to a given point on $\phi^{m}(\frac12\partial B)$ along any horizontal curve with bounded speed $|\dot x(t)|_g\le L$ is separated from $0$. A similar statement holds for the motion time from $\phi^{m}(\frac12\partial B)$ to $\phi^{m}(\frac23\partial B)$. Hence, there exists a constant $S^+>0$ such that, for any horizontal curve $x(t)$ in $\K$ with $|\dot x|_g\le L$, we have $s^+_j-t_j \ge S^+$ and $t_{j+1}-s_j^+\ge S^+$ for $j<N-1$. For $j<N-1$, let us now put $s^-_j=\max\{t\in[t_j;t_{j+1}]:x(t)\not\in \phi^{m_{j+1}}(\frac12B)\}$ if the written set is not empty, and $s^-_j=t_j$ otherwise. Similarly, there exists a constant $S^->0$ such that $t_{j+1}-s_j^-\ge S^-$ for any horizontal curve in $\K$ with $|\dot x|_g\le L$. Put $t_{j+1/2}=\max\{s_j^+,s_j^-\}$. Obviously,  $x(t)\in \phi^{m_j}(\frac23B)$ for $t\in[t_j;t_{j+1/2}]$ (as $[t_j;t_{j+1/2}]\subset [t_j;t_{j+1})$) and $x(t)\in \phi^{m_{j+1}}(\frac23B)$ for $t\in[t_{j+1/2};t_{j+1}]$ (as $x(t)\in \phi^{m_{j+1}}(\frac12B)$ for $t\in (s_{j}^-;t_{j+1}]$). Put $S=\min\{S^+,S^-\}$. Then $t_{j+1/2}-t_j\ge S$ and $t_{j+1}-t_{j+1/2}\ge S$ for $j<N-1$. In particular, number of points $N$ in the partition never exceeds $1+2/S$ on any horizontal curve $x(t)$ in $\K$ with $|\dot x(t)|_g\le L$.
	
	Consider the following system of overlapping intervals:
	\[
		[t_0;t_1],\ [t_{1/2};t_2],\ [t_{3/2},t_3],\ \ldots,\ [t_{N-3/2},t_{N}].
	\]
	We know that $x(t)\in \phi^{m_j}(\frac23\cl B)$ for $t\in[t_{j-1/2};t_{j+1}]$ and $x(t)\in\phi^{M_1}(\frac23\cl B)$ for $t\in[t_0;t_1]$ by construction. Moreover, each interval length is\footnote{Except for the last one $[t_{N-3/2},t_{N}]$, which length is $S$ or greater.} $2S$ or greater, and length of any two consequent intervals intersection is $S$ or greater. Therefore, $L_p$-H\"older continuity of $\dot x$ on $[0;1]$ with an exponent $\nu$ is equivalent to $L_p$-H\"older continuity of $\dot x$ on each of these intervals with the same exponent $\nu$. Moreover, for any horizontal curve $x(t)\in \K$ with $|\dot x(t)|_g\le L$, the norm of $\dot x$ in $B^\nu_{p,\infty}[0;1]$ is not greater than the sum of its norms in  $B^\nu_{p,\infty}[t_0,t_1]$ and $B^\nu_{p,\infty}[t_{j-1/2},t_{j+1}]$ for $j<N$.
	
	Thus, it is sufficient to prove Theorem~\ref{thm:main} for the finite number of manifolds $M_m=\phi^{x_m}(B)$ (with compacts $\K_m=\K\cap \phi^{m}(\frac23\cl B)$). Conditions~\eqref{oc:f_i} and~\eqref{oc:f_i_const} are fulfilled on any manifold $M_m$. It remains to take into account the renormalization of time $\tau=(t-t_{j-1/2})/(t_{j-1/2}-t_{j+1})$ (or $\tau=(t-t_0)/(t_1-t_0)$), which allows one to assume $[t_{j-1/2};t_{j+1}]=[0;1]$ (or $[t_0;t_1]=[0;1]$). If the interval length is $\sigma$, then this renormalization changes norm $\|u\|_p$ by a factor of $\sigma^{-p}$, and constant $c^\alpha_p(u)$ by a factor of $\sigma^{-p-\alpha}$. Hence, after the renormalization, we obtain a norm that is equivalent to the  norm for the Besov space on $[0;1]$ given in Definition~\ref{defn:Besov_space}, and this is of no importance as length $\sigma$ is bounded from $0$ by $S$ and does not exceed~1.
	
\end{proof}

\subsection{A series of simplifications}
\label{subsec:simplifications}

Assumptions~\eqref{oc:f_i} and~\eqref{oc:f_i_orthonormal} allow us to reformulate the sub-Riemannian problem on shortest paths as an optimal control problem of the following form:
\begin{equation}
\label{problem:sr_optimal_control}
	\begin{gathered}
		l(x)=\int_0^1\sqrt{u_1^2+\ldots+u_k^2}\to\min\\
		\dot x = u_1f_1(x)+\ldots+u_kf_k(x)\\
		x(0)=x^0;\qquad x(1)=x^1.
	\end{gathered}
\end{equation}
Here $x:[0;1]\to M$ is a Lipschitz continuous curve and $u=(u_1,\ldots,u_k)\in L^\infty([0;1]\to\R^k)$ is a bounded control. 

We believe that in the optimal control setting the term ``admissible'' is more appropriate than ``horizontal''. So pair $(x,u)$ is called admissible if it satisfies both ODE and terminal conditions in~\eqref{problem:sr_optimal_control}. All previously introduced notions can also be reformulated in the new setting, e.g.\ a pair $(x,u)$ is called $W^1_1$-local minimum if there exists an $\varepsilon>0$ such that, for any admissible pair $(y,v)$ with $\|u-v\|_1<\varepsilon$, one has $l(x)\le l(y)$.

Note that condition $|\dot x|_g=\const$ in Theorem~\ref{thm:main} is equivalent to the following one
\begin{equation}
\label{eq:u_square_eq_l}
	|\dot x(t)|_g = |u(t)| \equiv l(x)\ne 0
\end{equation}
(we assume $l(x)\ne 0$ since $u\equiv 0$ if $l(x)=0$ and there is nothing to be proved).

\medskip

We claim that to prove Theorem~\ref{thm:main} it is sufficient to prove that the control $u$ is $L_p$-H\"older continuous. Indeed, since $x\in W^1_\infty(0;1)$, $L_p$-H\"older continuity of $u$ easily implies $L_p$-H\"older continuity of $\dot x(t)=\sum_j f_j(x(t))u_j(t)$ (for item \cg, one should also use the fact that fields $f_j$ and their derivatives $\frac{\partial}{\partial x}f_j$ are bounded on the compact $\K$).

Thus, item \cw in Theorem~\ref{thm:main} states that, for any $W_1^1$-local minimum $(x,u)$ in problem~\eqref{problem:sr_optimal_control} with $|\dot x|_g=\const$ and for all $i\le k$, one has $\|u_i\|_{B^{\alpha}_{p,\infty}}<\infty$. Item \cg in the theorem states that there exists a constant $C>0$ such that, for any two given points $x^0,x^1\in \K$, any global minimum $(x,u)$ in problem~\eqref{problem:sr_optimal_control} with $x(t)\in \K$, and $|\dot x|_g=\const$ and for all $i\le k$, one has $\|u_i\|_{B^{\beta}_{p,\infty}}\le Cl(x)^\kappa$. Without loss of generality, we prove both items assuming $i=1$ since, for other $i\le k$, the proof is identical.

\medskip

Let us make some additional simplifications to problem~\eqref{problem:sr_optimal_control}. Without loss of generality, we assume that
\begin{equation}
\label{ass:f_1_const}
\tag{$A_1$}
	f_1(x)\equiv(1,0,\ldots,0)
\end{equation}

\begin{proof}[Proof of~\eqref{ass:f_1_const}.]
	
	Indeed, according to~\eqref{oc:f_i_const}, each vector filed $f_i$ is constant under an appropriate choice of global coordinates on $M$. Let us take the one that makes $f_1$ constant.
	
\end{proof}

Without loss of generality, we also assume that
\begin{equation}
\label{ass:f_bounded}
\tag{$A_2$}
	M=\R^n
	\quad\text{and}\quad
	\exists c_f>0\ :\ 
	\sup\left\{\|f_i\|_{C^0},\big\|\frac{\partial}{\partial x}f_i\big\|_{C^0},\big\|\frac{\partial^2}{\partial x^2}f_i\big\|_{C^0}\right\} \le c_f.
\end{equation}

\begin{proof}[Proof of~\eqref{ass:f_bounded}.] 
	
	According to~\eqref{oc:f_i_const}, $M$ is an open subset in $\R^n$. Field $f_1$ can be trivially prolonged to $\R^n$ by keeping condition~\eqref{ass:f_1_const}. For $f_i$ having $i\ge 2$, we use the following standard trick. Let $\K$ be embedded into $M$ with its $\varepsilon$-neighborhood. Choose a sufficiently smooth function $\psi:\R^n\to\R$ that satisfies $\psi\equiv 1$ on\footnote{Here we use the Minkowski sum $A+B=\{a+b|a\in A,b\in B\}$.} $\K+\frac13\varepsilon B$, $\psi\equiv 0$ outside $\K+\frac23\varepsilon B$, and $\psi(x)\in [0;1]$ for all $x$. Let us now exchange $f_i$ by~$\tilde f_i=\psi f_i\in C^\infty(\R^n)$ and~$M$ by~$\tilde M=\R^n$ in problem~\eqref{problem:sr_optimal_control}. Under such an exchange of vector fields, we have: (i) some admissible curves that leave $\K+\frac13\varepsilon B$ may become nonadmissible; (ii) lengths of remaining admissible curves may only become bigger; and (iii) lengths of admissible curves in $\K+\frac13\varepsilon B$ stay the same. Hence, any admissible curve in $\K$ can not lose any of properties of being global minimum or $W^1_1$-local minimum. Moreover, the H\"ormaner condition at points in $\K+\frac13\varepsilon B$ is not affected by this exchange.

\end{proof}

From now on, we denote the standard Euclidean norm of vector $x\in M=\R^n$ by $|x|$. Recall that $d_{SR}(x,y)$ denotes the sub-Riemannian distance between $x,y\in M=\R^n$ (if there is no admissible curves joining $x$ and $y$, we put $d_{SR}(x,y)=\infty$). So we claim that

\begin{equation}
\label{ass:bb_thm}
\tag{$A_3$}
	\exists\,\sigma,c_{bb}>0 : \forall x\in \K\ \forall y\in \K+\sigma B
	\qquad
	\frac1{c_{bb}}|x-y|\le d_{SR}(x,y)\le c_{bb}|x-y|^{1/s}.
\end{equation}

\begin{proof}[Proof of~\eqref{ass:bb_thm}.]
	
	It follows directly from the ball-box theorem~\cite{Gromov} and the fact that step of the sub-Riemannian manifold at points in $\K$ never exceeds $s$ by the hypothesis of the theorem.
	
\end{proof}

So, without loss of generality, we assume that conditions~\eqref{ass:f_1_const}, \eqref{ass:f_bounded}, and~\eqref{ass:bb_thm} are fulfilled for optimal control problem~\eqref{problem:sr_optimal_control}.

\subsection{An interpolational estimate on the optimal control}
\label{subsec:interpolation}

We start with a brief explanation of key ideas the proof is based on. Consideration of the interpolational estimate on the optimal control is based on the following well known idea. Denote by $\dot u_1$ the (generalized) derivative of $u_1$. Suppose for a moment that $\dot u_1\in L_r(0;1)$ for some $1<r<\infty$. Then, for any smooth test function $\varphi\in C_0^\infty(0;1)=\{\varphi\in C^\infty:\mathrm{supp}\,\varphi\Subset(0;1)\}$, integration by parts together with the H\"older inequality imply
\[
	\int_0^1 u_1\dot\varphi\,dt=-\int_0^1 \dot u_1\varphi\,dt\le \|\dot u\|_r\|\varphi\|_{r^*},
\]
where $rr^*=r+r^*$. The reversed statement also holds: if $\int_0^1u_1\dot\varphi\,dt\le\const\|\varphi\|_{r^*}$ for any test function $\varphi\in C_0^\infty(0;1)$, then the generalized derivative $\dot u_1$ is a bounded linear functional on a dense subset of $L_{r^*}(0;1)$ and, hence, $\dot u_1\in (L_{r^*}(0;1))^*=L_r(0;1)$. So in the latter case, $u_1$ belongs to the Sobolev space~$W^1_r(0;1)$.

At the same time, for an arbitrary function $v\in L_q(0;1)$, we are only able to guarantee that $\int_0^1 v\dot\varphi\,dt\le \|u\|_q\|\dot\varphi\|_{q^*}$ and one can not expect $\|\varphi\|_{r^*}$ (instead of $\|\dot\varphi\|_{q^*}$) in the right-hand side of the inequality. The proof of Theorem~\ref{thm:main} is based on an interpolational estimate on the optimal control: we will prove that
\[
	\int_0^1 u_1\dot\varphi\,dt \le C\|\varphi\|_r^{1-\theta}\|\dot\varphi\|_q^\theta
\]
for some $0<\theta<1$ and (not necessarily conjugate) $q$ and $r$. This estimate appears to be dual to an estimate on $K$-functional of $\theta$-interpolation of spaces $L_{q^*}(0;1)$ and $W^1_{r^*}(0;1)$ and allows us to prove that $u_1$ is $L_p$-H\"older continuous by proceeding to the Besov spaces.

\medskip

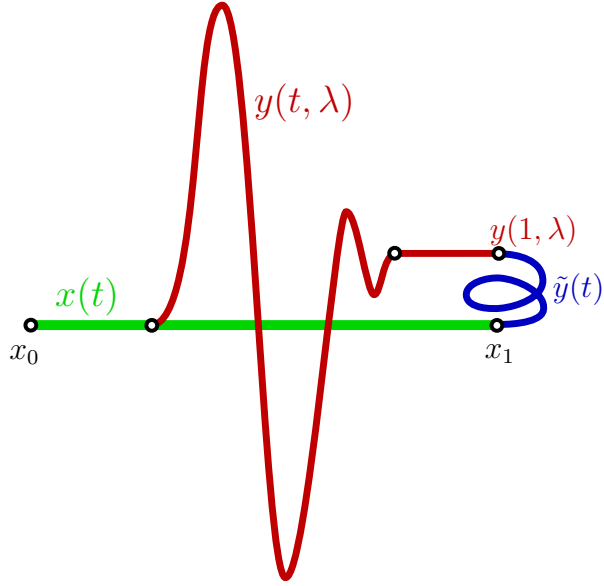
\begin{figure}[ht]\centering
	\def\svgwidth{0.4\textwidth}

	\begingroup%
	\makeatletter%
	\providecommand\color[2][]{%
		\errmessage{(Inkscape) Color is used for the text in Inkscape, but the package 'color.sty' is not loaded}%
		\renewcommand\color[2][]{}%
	}%
	\providecommand\transparent[1]{%
		\errmessage{(Inkscape) Transparency is used (non-zero) for the text in Inkscape, but the package 'transparent.sty' is not loaded}%
		\renewcommand\transparent[1]{}%
	}%
	\providecommand\rotatebox[2]{#2}%
	\newcommand*\fsize{\dimexpr\f@size pt\relax}%
	\newcommand*\lineheight[1]{\fontsize{\fsize}{#1\fsize}\selectfont}%
	\ifx\svgwidth\undefined%
	\setlength{\unitlength}{443.53245678bp}%
	\ifx\svgscale\undefined%
	\relax%
	\else%
	\setlength{\unitlength}{\unitlength * \real{\svgscale}}%
	\fi%
	\else%
	\setlength{\unitlength}{\svgwidth}%
	\fi%
	\global\let\svgwidth\undefined%
	\global\let\svgscale\undefined%
	\makeatother%
	\begin{picture}(1,1.03662029)%
		\lineheight{1}%
		\setlength\tabcolsep{0pt}%
		\put(0,0){\includegraphics[width=\unitlength,page=1]{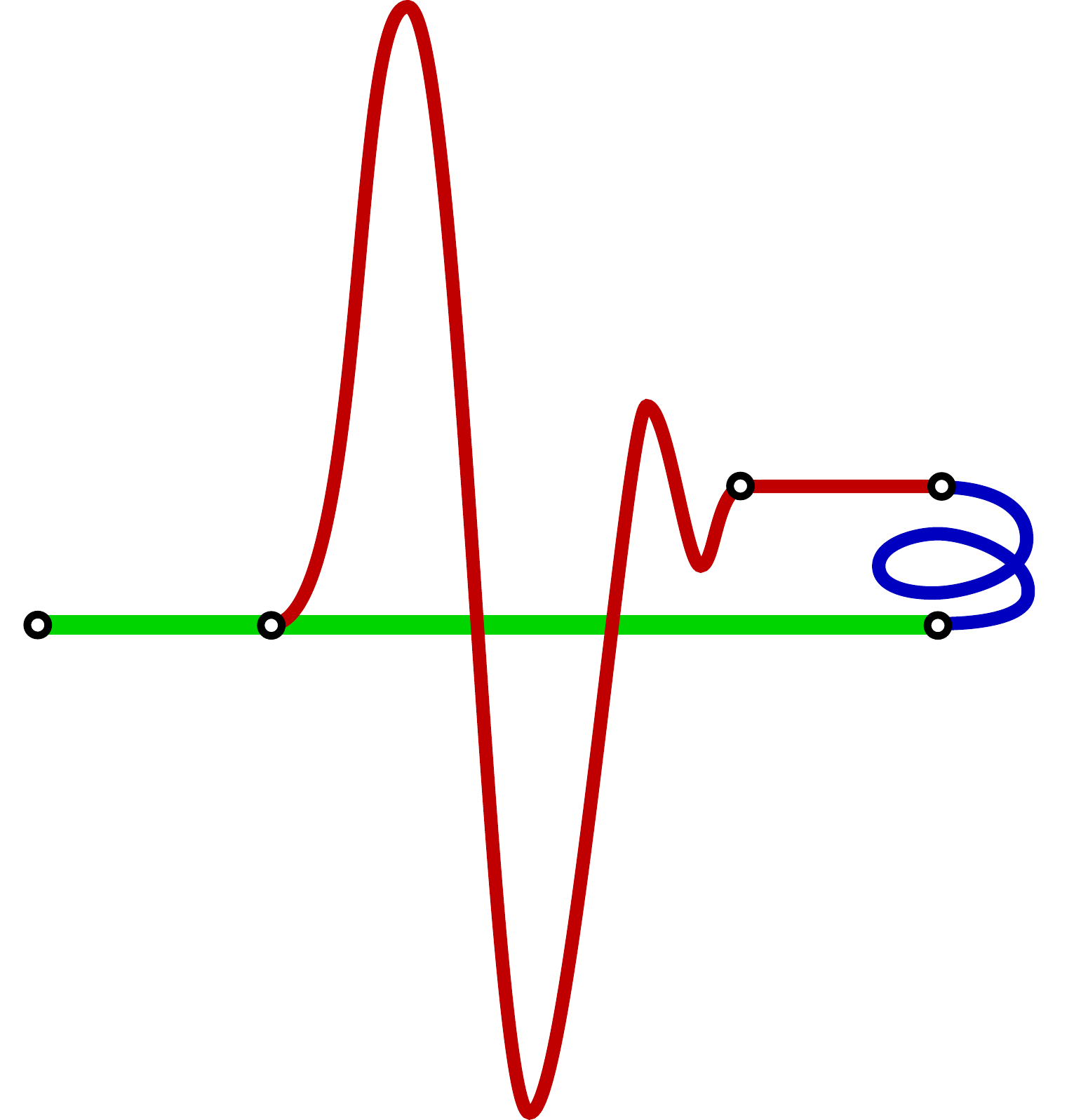}}%
		\put(-0.0031935,0.39745752){\color[rgb]{0,0,0}\makebox(0,0)[lt]{\lineheight{1.25}\smash{\begin{tabular}[t]{l}$x_0$\end{tabular}}}}%
		\put(0.85530544,0.61289639){\color[rgb]{0.75294118,0,0}\makebox(0,0)[lt]{\lineheight{1.25}\smash{\begin{tabular}[t]{l}$y(1,\lambda)$\end{tabular}}}}%
		\put(0.8464639,0.39843477){\color[rgb]{0,0,0}\makebox(0,0)[lt]{\lineheight{1.25}\smash{\begin{tabular}[t]{l}$x_1$\end{tabular}}}}%
		\put(0.96811086,0.50783625){\color[rgb]{0,0,0.75294118}\makebox(0,0)[lt]{\lineheight{1.25}\smash{\begin{tabular}[t]{l}$\tilde y(t)$\end{tabular}}}}%
		\put(0.43402762,0.83963757){\color[rgb]{0.75294118,0,0}\makebox(0,0)[lt]{\lineheight{1.25}\smash{\begin{tabular}[t]{l}\scalebox{1.2}{$y(t,\lambda)$}\end{tabular}}}}%
		\put(0.07864547,0.49270115){\color[rgb]{0,0.83529412,0}\makebox(0,0)[lt]{\lineheight{1.25}\smash{\begin{tabular}[t]{l}\scalebox{1.2}{$x(t)$}\end{tabular}}}}%
	\end{picture}%
	\endgroup%

	\caption{Variation $y(t)$ of the original curve $x(t)$. On the figure, the end points of $\mathrm{supp}\,\varphi$ are also shown as dots for convenience.}
	\label{fig:heart}
\end{figure}

The heart of the proof is the following variation depicted on Fig.~\ref{fig:heart}. Let us choose an arbitrary test function $\varphi\in C_0^\infty(0;1)$ and consider a new control\footnote{The condition $\lambda\ge 0$ is of no importance and introduced by the authors to simplify further computations.} $u_1-\lambda\dot\varphi$ (in the place of $u_1$), $\lambda\ge 0$, and the corresponding trajectory $y(t,\lambda)$ starting at $x^0$, i.e.\ $y(\cdot,\lambda)$ is the unique solution to the following ODE
\[
	\dot y = (u_1-\lambda\dot\varphi)f_1 + u_2f_2+\ldots u_kf_k
\]
with the initial state $y(0,\lambda)=x^0$. Since $\varphi(0)=\varphi(1)=0$, we have $\int_0^1\dot\varphi\,dt=0$. Therefore, no matter how large modulo values the function $\dot\varphi$ takes, its fluctuations on the interval~$[0;1]$ will compensate each other in the variation $y(t,\lambda)$, and the final distance between $x^1$ and $y(1,\lambda)$ will be estimated in terms of the norm of function $\varphi$ itself, and not by its derivative $\dot\varphi$ (see Lemma~\ref{lm:xt_yt_estimate}). Further, if $x^1\ne y(1,\lambda)$, then we add an admissible curve $\tilde y(t)$ that joins $y(1,\lambda)$ and $x^1$. Length of that curve can be controlled by the ball-box Theorem (see Lemma~\ref{lm:basic_u_dot_phi_estimate}). At the end, we compare lengths of the original curve and the resulting variation, which gives us the desired interpolational estimate on the control $u_1$.

\medskip

So, let us denote by $P(t,\lambda)$ the flow of the linear part of the variational equation along $y$, i.e.~$P(t,\lambda)\in\mathrm{GL}(n)$ and
\[
	\dot P = \big(u_2\,\frac{\partial}{\partial x}{f_2} + \ldots + u_k\,\frac{\partial}{\partial x}f_k\big)P
	\quad\text{and}\quad
	P(0,\lambda)=1,
\]
and put $P_0(t)=P(t,0)$ for short. The term $u_1\,\frac{\partial}{\partial x}f_1$ vanish since $f_1$ is constant.

As we have mentioned above, the first key idea in the theorem proof is the following: change of the original trajectory is estimated by the norm of $\varphi$ itself, and not by the norm of its derivative $\dot\varphi$.

\begin{lemma}
\label{lm:xt_yt_estimate}
	
	There exists a constant\footnote{Which does not depend on $\varphi$, $\lambda$, points $x^0,x^1\in \K$ and pair $(x,u)$ with $x(t)\in \K$.} $c>0$ such that for all $\lambda\ge 0$ and $\varphi\in C_0^\infty(0;1)$, one has
	\[
		|y(t,\lambda)-x(t)|\le \lambda\big(|\varphi(t)|+c\,l(x)\|\varphi\|_1\big).
	\]
	If additionally test function $\varphi$ satisfies the following equality\footnote{As usual, here we compute values of brackets $[f_1,f_j]$ at the corresponding point of the original curve $x(t)$.}
	\begin{equation}
	\label{eq:main_varphi_restriction}
		\int_0^1 \varphi P_0([f_1,f_2]u_2+\ldots+[f_1,f_k]u_k)\,dt = 0,
	\end{equation}
	then
	\begin{equation}
	\label{eq:second_variation_estimate}
		|y(1,\lambda)-x^1|\le c\lambda^2\|\varphi\|_2^2.
	\end{equation}
	
\end{lemma}

\begin{proof}
	
	Let us estimate norms of operator $P\in\mathrm{GL}(n)$ first. Assumption~\eqref{ass:f_bounded} implies that $|\frac{\partial}{\partial x} f_j|\le c_f$ and $\|u_j\|_\infty\le L$. Therefore, using the Gr\"onwall inequality, we obtain $\|P\|\le e^{kLc_f}$. Similarly, $\|P^{-1}\|\le e^{kLc_f}$. In other words, norms of operators $P$ and $P^{-1}$ are bounded above by constant~$e^{kLc_f}$ for all $t$, $\lambda$, and $\varphi$.
	
	Further, using the theorem on smooth dependence of ODE solutions on the parameter, we obtain that $y(t,\lambda)$ is a.e.\ differentiable in $t$ and smooth in $\lambda$. Thus, according to the Lagrange formula, we have
	\[
		|y(t,\lambda) - x(t)| \le \lambda \sup_{\mu\in[0;\lambda]}|y_\lambda(t,\mu)|
	\]
	Let us compute $y_\lambda$. Assumption~\eqref{ass:f_1_const} implies
	\[
		\dot y_\lambda = \big(u_2\,\frac{\partial}{\partial x}f_2 + \ldots + u_k\, \frac{\partial}{\partial x}f_k\big)y_\lambda - \dot\varphi f_1.
	\]

	Hence, using formula for the general solution to affine ODE, we obtain
	\[
		y_\lambda(t,\lambda) = -P(t,\lambda)\int_0^t \dot\varphi(s)P^{-1}(s,\lambda) f_1(y(s))\,ds,
	\]
	which, after an integration by parts, becomes
	\[
		y_\lambda(t,\lambda) = -\varphi(t)f_1(y(t)) + P(t,\lambda)\int_0^t\varphi P^{-1}([f_2,f_1]u_2+\ldots+[f_k,f_1]u_k)\,ds
	\]
	since $f_1$ is constant and $\varphi(0)=0$. Using $|u(t)|\equiv l(x)$, we have
	\[
		|y_\lambda(t,\lambda)|\le |\varphi(t)| + kLc_fe^{2kLc_f}l(x)\|\varphi\|_1,
	\]
	which proves the first estimate given in the lemma statement.
	
	Now we proceed to the second estimate. Here we restrict ourselves to the right end-point. So let us write down the Taylor expansion up the second order:
	\[
		|y(1,\lambda) - x(1) - \lambda y_\lambda(1,0)| \le \frac12\lambda^2 \sup_{\mu\in[0;\lambda]}|y_{\lambda\lambda}(1,\mu)|
	\]
	and
	\[
		\dot y_{\lambda\lambda} = \big(u_2\frac{\partial}{\partial x} f_2 + \ldots + u_k \frac{\partial}{\partial x}f_k\big)y_{\lambda\lambda} +
			\big(u_2 \frac{\partial^2}{\partial x^2}f_2 + \ldots + u_k \frac{\partial^2}{\partial x^2}f_k\big)(y_\lambda,y_\lambda).
	\]
	
	If $\varphi$ satisfies the linear equation~\eqref{eq:main_varphi_restriction}, then the linear part in the Taylor expansion disappears, i.e.\ $y_\lambda(1,0)=0$. Hence, if we consider only test functions $\varphi\in C^\infty(0;1)$ that satisfy~\eqref{eq:main_varphi_restriction}, then position of the right end $y(1,\lambda)$ is determined by the second variation. Let us estimate it. Since
	\[
		y_{\lambda\lambda}(1,\lambda) = P(1,\lambda) \int_0^1 P^{-1}\big(u_2\,\frac{\partial^2}{\partial x^2}f_2 + \ldots + u_k \,\frac{\partial^2}{\partial x^2}f_k\big)(y_\lambda,y_\lambda)\,dt,
	\]
	we obtain
	\[
		|y_{\lambda\lambda}(1,\lambda)| \le kLc_fe^{2kc_f}\int_0^1|y_\lambda|^2\,dt \le
		kLc_fe^{2kc_f}(\|\varphi\|_2^2 + 2kLc_fe^{kLc_f}\|\varphi\|_1^2 + k^2L^2c_f^2e^{2kLc_f}\|\varphi\|_1^2).
	\]
	Obviously $\varphi\in L_2(0;1)$ implies $\|\varphi\|_1\le\|\varphi\|_2$ by the H\"older inequality. Hence,
	\[
		|y_{\lambda\lambda}(1,\lambda)| \le kLc_fe^{2kc_f} (1 + 2kLc_fe^{kLc_f} + k^2L^2c_f^2e^{2kLc_f})\|\varphi\|_2^2.
	\]

	Thus,
	\[
		c=\max\Big\{1, kLc_fe^{kLc_f},kLc_fe^{2kLc_f}(1 + 2kLc_fe^{kLc_f} + k^2L^2c_f^2e^{2kLc_f})\Big\}.
	\]
\end{proof}	

In the following lemma, we prove that, for some values of parameters $r$, $\zeta$, and $c'$, the following implication is fulfilled
\begin{equation}
\label{eq:u_1_dot_phi_base_estimate}
	c'\,l(x)\lambda\|\varphi\|_1\le 1\\
	\quad\Longrightarrow\quad
	l(y)-l(x)+ 
	c'(l(x)\lambda\,\|\varphi\|_r)^\zeta\ge 0.
\end{equation}
Note that, using this implication, we then obtain a dual interpolational estimate on the control $u_1$ in Proposition~\ref{prop:interpolation_estimate}.

\begin{lemma}
\label{lm:basic_u_dot_phi_estimate}

	There exist

	\begin{itemize}

		\item[\cg] a constant $\hat c'$ such that, for any sub-Riemannian shortest path $(x,u)$ with $x(t)\in \K$, the material implication~\eqref{eq:u_1_dot_phi_base_estimate} holds for all $\lambda\ge 0$ and functions $\varphi\in C_0^\infty(0;1)(0;1)$ if $r=1$, $\zeta=1/s$, and $c'=\hat c'$;
		
		\item[\cw] a depending on the choice of $W^1_1$-local minimum $(x,u)$ constant $\tilde c'=\tilde c'(x,u)$ such that the material implication~\eqref{eq:u_1_dot_phi_base_estimate} holds for all $\lambda\ge 0$ and functions $\varphi\in C_0^\infty(0;1)(0;1)$ if $r=2$, $\zeta=2/s$, $c'=\tilde c'$, and $\varphi$ and $\lambda$ satisfy additional restriction $c'\,l(x)\lambda\|\dot\varphi\|_1\le 1$.
		
	\end{itemize}
\end{lemma}

\begin{proof}
	
	Let us use~\eqref{ass:bb_thm}. By the hypothesis of Theorem~\ref{thm:main}, $x^1\in \K$. Hence, Lemma~\ref{lm:xt_yt_estimate} implies that if $|x^1-y(1,\lambda)|<\sigma$, then
	\[
		d_\lambda \eqdef d_{SR}(x^1,y(1,\lambda))\le c_{bb}|x^1-y(1,\lambda)|^{1/s}.
	\]
	Let us consider only $\lambda\ge0$ such that $c\,l(x)\lambda\|\varphi\|_1\le \sigma$. For these $\lambda$, using Lemma~\ref{lm:xt_yt_estimate}, we obtain
	\begin{equation}
	\label{eq:d_lambda_estimate}
		d_\lambda\le 
		c_{bb}\left(c\,l(x)\lambda\|\varphi\|_1\right)^{1/s}.
	\end{equation}
	If additionally $c\,l(x)(c_{bb})^s\lambda\|\varphi\|_1\le 4^{-s}$, then $d_\lambda\le \frac14$. So, we put
	\[
		c'_0=c\max\{1/\sigma, (4c_{bb})^s\}
	\]
	and assume both $\hat c\ge c'_0$ and $\tilde c\ge c'_0$.

	By the definition of the sub-Riamennian distance, there exists a pair $(\tilde y,\tilde v)$ that satisfies the ODE in problem~\eqref{problem:sr_optimal_control} with terminal conditions $\tilde y(0)=y(1,\lambda)$, $\tilde y(1)=x^1$ (see Fig.~\ref{fig:heart}) and has length $\rho\le 2d_\lambda\le \frac12$. Without loss of generality, we assume $|\tilde v(t)|\equiv \rho$. Consider the following composite pair $(z,w)$, $t\in[0;1]$:
	\[
		z(t)=\begin{cases}
			y(t/(1-\rho)),&\text{for } t\in[0;1-\rho];\\
			\tilde y((t-1+\rho)/\rho),&\text{for }t\in[1-\rho;1];
		\end{cases}
		\quad
		w(t)=\begin{cases}
			\frac1{1-\rho} v(t/(1-\rho)),&\text{for } t\in[0;1-\rho);\\
			\frac1\rho\tilde v((t-1+\rho)/\rho),&\text{for }t\in(1-\rho;1];
		\end{cases}
	\]
	where $v = u-\lambda(\dot\varphi,0\ldots,0)$. Now, the pair $(z,w)$ does satisfy both the ODE and terminal conditions in problem~\eqref{problem:sr_optimal_control}.
	
	First, we suppose that $(x,u)$ is a global minimum for some given points $x^0$ and $x^1$. In this case, if the estimate $\lambda\|\varphi\|_1\le 1/(\hat c'\,l(x))$ holds, then length of the original curve $x(t)$ cannot be greater than length of the composite curve~$z(t)$. Therefore, using Lemma~\ref{lm:xt_yt_estimate}, we obtain
	\begin{equation}
	\label{eq:temp_dot_u_estimate}
		l(x) \le l(z) \le 
		l(y) + 2d_\lambda \le
		l(y) + 2c_{bb}\,(c\lambda\,l(x)\|\varphi\|_1)^{1/s}.
	\end{equation}
	Thus, item \cg in the lemma is proved for $\hat c'=\max\{c'_0,c(c_{bb})^s\}=c'_0$.
	
	\medskip
	
	Second, we prove item \cw in the Lemma. Let $(x,u)$ be not necessarily global but $W^1_1$-local minimum in problem~\eqref{problem:sr_optimal_control} for some given $x^0$ and $x^1$. In this case, we are able to find a constant $\tilde c'$ (stated in the lemma) for any given pair $(x,u)$, but $\tilde c'$ highly depends on $(x,u)$ in general, and it seems that $\tilde c'$ cannot be chosen uniformly among all $W^1_1$-local minima in $\K$.
	
	Summarizing, let us fix a $W^1_1$-local minimum $(x,u)$ for the rest of the lemma proof. We claim that, for small enough $\lambda$, the composite control $w(t)$ belongs to the $\varepsilon$-neighborhood\footnote{Here $\varepsilon$ is taken from the definition of $W^1_1$-local minimum.} of the original control $u(t)$, i.e.\ $\|u-w\|_1<\varepsilon$ (which then implies that the composite pair $(z,w)$ length cannot be less than length of $(x,u)$).
	
	In order to prove inequality $\|u-w\|_1<\varepsilon$, let us estimate the left-hand side:
	\[
		\|u-w\|_1\le\|u-v\|_1 + 
		\int_0^{1-\rho}\left|v(t)-\frac{1}{1-\rho}v(t/(1-\rho))\right|\,dt + 
		\int_{1-\rho}^1|v(t)|\,dt+
		\frac1\rho\int_{1-\rho}^1|\tilde v((t-1+\rho)/\rho)|\,dt.
	\]
	Now we estimate each term in the right-hand side separately:
	\[
		\begin{gathered}
			\|u-v\|_1 = \|(\lambda\dot \varphi,0,\ldots,0)\|_1=\lambda\|\dot\varphi\|_1;\\
			\int_0^{1-\rho}\left|v(t)-\frac1{1-\rho}v(t/(1-\rho))\right|\,dt \le 2\lambda\|\dot\varphi\|_1+\sum_{i=1}^k\int_0^1|(1-\rho)u_i((1-\rho)s)-u_i(s)|\,ds;\\
			\int_{1-\rho}^1|v(t)|\,dt \le \lambda\|\dot\varphi\|_1+\sum_{i=1}^k\int_{1-\rho}^1|u_i(t)|\,dt\le \lambda\|\dot\varphi\|_1 + kL\rho;\\
			\frac1\rho\int_{1-\rho}^1|\tilde v((t-1+\rho)/\rho)|\,dt = \rho.
		\end{gathered}
	\]
	Obviously, $L_1(0;1)$-norm of each function $(1-\rho)u_i((1-\rho)t)-u_i(t)$ tends to $0$ as $\rho\to+0$. Therefore,
	\[
		\|u-w\|_1 \le 4\lambda\|\dot\varphi\|_1 + \nu(\rho)
	\]
	where $\nu(\rho)\ge 0$ and $\lim_{\rho\to+0}\nu(\rho)=0$. Since $\rho\le d_\lambda$, there exists $\hat d>0$ such that $\nu(\rho)<\varepsilon/2$ if $d_\lambda\le\hat d$.
	
	Thus, in order to fulfill inequality $\|u-w\|_1<\varepsilon$, it is enough to have $4\lambda\|
	\dot\varphi\|_1\le \varepsilon/2$ and $c_{bb}\left(c\,l(x)\lambda\|\dot\varphi\|_1\right)^{1/s}\le \hat d$ due to~\eqref{eq:d_lambda_estimate}. The last two inequalities follow from the hypothesis in item \cw of the lemma and the condition in material implication~\eqref{eq:u_1_dot_phi_base_estimate} if
	\[
		\tilde c'l(x)\ge c'_1=\max\{8/\varepsilon,c\,l(x)(c_{bb}/\hat d)^s\}.
	\]
	
	So let us choose $\tilde c'\ge\max\{c'_0,c'_1/l(x)\}$ (recall that $l(x)\ne 0$ due to~\eqref{eq:u_square_eq_l}). Hence, if estimate $\tilde c'\,l(x)\lambda\|\dot\varphi\|_1\le 1$ holds, then $\|u-w\|<\varepsilon$, and length of the original curve $x(t)$ cannot be less than length of the composite curve $z(t)$: $l(x)\le l(y)+2d_\lambda$. 
	
	Therefore, we obtain the same result as in item \cg, but now we are able to use more accurate estimate on the distance between $x^1$ and $y(1,\lambda)$, which is given in the second part of Lemma~\ref{lm:xt_yt_estimate}:
	\[
		d_\lambda \le c_{bb}|x^1-y(1,\lambda)|^{1/s} \le c_{bb}c^{1/s}\lambda^{2/s}\|\varphi\|^{2/s}_2.
	\]
	Thus, item \cw of the Lemma is proved for $\tilde c'\ge c_{bb}(c/l(x)^2)^{1/s}$. In other words, item \cw of the lemma is proved by taking
	\[
		\tilde c' = \max\{c'_0,c'_1/l(x),c_{bb}(c/l(x)^2)^{1/s}\}.
	\]
\end{proof}

In order to obtain the interpolational estimate on the control $u_1$, we construct a lower bound for the left-hand side of~\eqref{eq:u_1_dot_phi_base_estimate}. Vector fields $f_j$ are orthonormal due to assumption~\eqref{oc:f_i_orthonormal}. Hence,
\begin{equation}
\label{eq:l_y_sqrt}
	l(y) = \int_0^1 \sqrt{(u_1-\lambda\dot\varphi)^2+u_2^2+\ldots+u_k^2}\,dt=\int_0^1\sqrt{l(x)^2-2\lambda u_1\dot\varphi+\lambda^2\dot\varphi^2}\,dt.
\end{equation}
Note that the previous equality is the only place in the proof of Theorem~\ref{thm:main} where we are actually use assumption~\eqref{oc:f_i_orthonormal}. It has not been used before and will not be used further.

It is not so hard to estimate $l(y)$ with the help of the following lemma.

\begin{lemma}
\label{lm:sqrt_estimate}

	Let $l>0$ and $1\le q\le 2$. Then, there exists a number $c_{SR}>0$ such that, for any $\psi\in C_0^\infty(0;1)$ and $u_1\in L_\infty(0;1)$ having $\|u_1\|_\infty\le l$, the following inequality is fulfilled
	\[
		\int_0^1\sqrt{l^2-2u_1\psi+\psi^2}\,dt -l \le c_{SR}l^{1-q}\|\psi\|_q^{q} - \frac1l\int_0^1u_1\psi\,dt.
	\]
	
\end{lemma}

\begin{proof}

	First, let $l=1$. Then
	
	\begin{itemize}
		
		\item If $|\psi(t)|\le 4$, then we are able to estimate the square root in standard manner:
		\[
			\sqrt{1-2 u_1\psi+\psi^2} - 1 \le 
			- u_1\psi + \frac12\psi^2\le - u_1\psi + \psi^2.
		\]
		
		\item If $|\psi(t)|>4$, then $|u_1\psi|\le |\psi|\le \frac14\psi^2$ (as $|u_1|\le 1$) and $1\le\frac1{16}\psi^2$. Hence,
		\[
			\sqrt{1-2 u_1\psi+\psi^2} - 1 \le
			2|\psi|\le
			- u_1\psi + 4|\psi|.
		\]
		
	\end{itemize}
	
	Denote the left-hand side of the inequality in the lemma statement by~$F(l,u_1,\psi)$. Then
	\begin{equation}
	\label{eq:without_sqrt_u_dot_phi_estimate}
		F(1,u_1,\psi)\le  
		\int\limits_{t:|\psi|\le 4}\psi^2\,dt +
		4 \int\limits_{t:|\psi|> 4}|\psi|\,dt -
		\int_0^1 u_1\psi\,dt.
	\end{equation}
	
	Now we unite both terms in the right-hand side of~\eqref{eq:without_sqrt_u_dot_phi_estimate} in the following way. Let $1\le q\le2$ be an arbitrary number. Then, for the first integral, we have
	\[
		\int\limits_{t:|\psi|\le 4}\psi^2\,dt = 
		\int\limits_{t:|\psi|\le 4}
		|\psi|^{2-q}|\psi|^q\,dt \le
		4^{2-q} \int\limits_{t:|\psi|\le 4}|\psi|^q\,dt.
	\]
	For the second integral, we have
	\[
		4\int\limits_{t:|\psi|>4}|\psi|\,dt=
		4^{2-q}\int\limits_{t:|\psi|>4}4^{q-1}|\psi|\,dt\le
		4^{2-q}\int\limits_{t:|\psi|>4}
		|\psi|^{q-1}|\psi|\,dt =
		4^{2-q} \int\limits_{t:|\psi|>4} |\psi|^q\,dt.
	\]	
	Thus, inequality~\eqref{eq:without_sqrt_u_dot_phi_estimate} implies
	\[
		F(1,u_1,\varphi) \le
		4^{2-q} \int_0^1 |\psi|^q\,dt -
		\int_0^1u_1\psi\,dt.
	\]
	So
	\[
		F(1,u_1,\psi) \le 4 \|\psi\|_q^q - \int_0^1u_1\psi.
	\]
	In other words, putting $c_{SR}=4$ proves the lemma in the case $l=1$.
	
	For an arbitrary $l$, the lemma follows from the fact that $F$ is homogeneous: $F(l,u_1,\psi) = lF(1,u_1/l,\psi/l)$.

\end{proof}

Let us now combine the previous two lemmas and show that they imply the following key interpolational estimate on the control $u_1$ for some $\zeta>0$:
\begin{equation}
\label{eq:u_1_dot_phi_middle_estimate}
	\begin{gathered}
		\int_0^1 u_1\dot\varphi\,dt\le
			c''l(x)^\kappa\|\varphi\|_r^\theta\|\dot\varphi\|_q^{1-\theta},\\
		\text{where}\quad
		\theta = \frac{(q-1)\zeta}{q-\zeta}
		\quad\text{and}\quad
		\kappa = \frac{1+2q\zeta-3\zeta}{q-\zeta}.
	\end{gathered}		
\end{equation}
Here $\varphi$ is an arbitrary test function that vanishes at $t=0,1$, namely
\[
	\varphi\in \mathring W^1_q(0;1)=
	\left\{\varphi\in W^1_q(0;1) : \varphi(0)=\varphi(1) \right\}.
\]

\begin{proposition}[interpolational estimate on control]
\label{prop:interpolation_estimate}	
\hspace{2em}
	\begin{itemize}

		\item[\cg] There exists a constant $\hat c''$ such that, for any $1\le q\le 2$, any sub-Riemannian shortest path $(x,u)$ with $x(t)\in \K$, and any test function $\varphi\in\mathring W^1_q(0;1)$, inequality~\eqref{eq:u_1_dot_phi_middle_estimate} holds for $c''=\hat c''$, $r=1$, and $\zeta=1/s$;
		
		\item[\cw] For any $1\le q\le 2$, there exists a constant $\tilde c''$ (which depends on the choice of $W^1_1$-local minimum $(x,u)$ in general, $\tilde c''=\tilde c''(x,u,q)$) such that, for any test function $\varphi\in\mathring W^1_q(0;1)$, inequality~\eqref{eq:u_1_dot_phi_middle_estimate} holds for $c''=\tilde c''$, $r=2$, and $\zeta=2/s$.
		
	\end{itemize}
\end{proposition}

\begin{proof}

	If $q=1$, then inequality~\eqref{eq:u_1_dot_phi_middle_estimate} is trivially fulfilled for~$c''=1$. So we assume that~$q>1$.

	First, let $\varphi$ be a smooth function, i.e.\ assume that $\varphi\in C_0^\infty(0;1)$. If we substitute expression~\eqref{eq:l_y_sqrt} for $l(y)$ and the estimate in Lemma~\ref{lm:sqrt_estimate} (with $\psi=\lambda\dot\varphi$ and $l=l(x)$) into the inequality in~\eqref{eq:u_1_dot_phi_base_estimate}, we then obtain
	\begin{equation}
	\label{eq:dot_u_middle_estimate}
		\lambda\int_0^1 u_1\dot\varphi\,dt\le c_{SR}\lambda^{q} l(x)^{2-q}\|\dot\varphi\|_q^q + c'\lambda^{\zeta}l(x)^{1+\zeta}\|\varphi\|_r^{\zeta}.
	\end{equation}

	Let us divide~\eqref{eq:dot_u_middle_estimate} by $\lambda$. Two terms will be obtained in the rhs: one has factor $\lambda^{q-1}$ and the other has factor $\lambda^{\zeta-1}$. The first one tend to $\infty$ as $\lambda\to+\infty$ and vanish as $\lambda\to0$. The second one tends to $\infty$ as $\lambda\to 0$ and vanish as $\lambda\to\infty$. So an optimal choice of $\lambda$ is a minimum point of the obtained rhs (which obviously exists due to described behavior of factors $\lambda^{q-1}$ and $\lambda^{\zeta-1}$). For simplicity put
	$\lambda_0=l(x)^{\frac{q+\zeta-1}{q-\zeta}}\|\varphi\|_r^{\frac\zeta{q-\zeta}}\|\dot\varphi\|_q^{-\frac{q}{q-\zeta}}$ which differs from the optimal value of $\lambda$ by a complicated constant factor of no importance. So we plug $\lambda=\lambda_0$ into~\eqref{eq:dot_u_middle_estimate}:
	\[
		\int_0^1 u_1\dot\varphi \le
			(c'+c_{SR})l(x)^\kappa\|\varphi\|_r^\theta\|\dot\varphi\|_q^{1-\theta}.
	\]

	Thus, in the case \cg, the desired inequality is proved for $\varphi\in C_0^\infty(0;1)(0;1)$ under restriction $c'l(x)\lambda_0\|\varphi\|_1\le 1$. In the case \cw, restrictions~\eqref{eq:main_varphi_restriction} and $c'l(x)\lambda_0\|\dot\varphi\|_1\le1$ should be also fulfilled. 
	
	First, let us throw away restriction $c'l(x)\lambda_0\|\varphi\|_1\le 1$. On the one hand, since
	\[
		c'\lambda_0l(x)\|\varphi\|_1 =
		c'l(x)^{\frac{2q-1}{q-\zeta}}
		(\|\varphi\|_1/\|\dot\varphi\|_q)^{\frac{q}{q-\zeta}},
	\]
	the restriction $c'\,l(x)\lambda_0\|\varphi\|_1\le 1$ is definitely fulfilled if
	\[
		c''_0
		l(x)^{\frac{2q-1}q}
		\|\varphi\|_1\le \|\dot\varphi\|_{L_{q}}
	\]
	where $c''_0\ge(c')^{\frac{q-\zeta}{q}}$. Since $1\le q\le 2$, we may put $c''_0=\max\{(c')^{1-\frac\zeta{2}},(c')^{1-\zeta}\}$. On the other hand, if the opposite inequality is fulfilled for some test function $\varphi$, then $\|u_1\|_\infty\le l(x)$ implies
	\[
		\begin{aligned}
			\int_0^1u_1\dot\varphi\,dt \le&
			l(x)\|\dot\varphi\|_1\le 
			l(x)\|\dot\varphi\|_q\le
			l(x)(c''_0l(x)^{2-\frac1q}
				\|\varphi\|_1)^\theta\|\dot\varphi\|_q^{1-\theta}= \\
			=&(c''_0)^\theta l(x)^{\kappa + \frac{q-1}{q}}
				\|\varphi\|_1^\theta\|\dot\varphi\|_q^{1-\theta}\le
			c''_1l(x)^\kappa
				\|\varphi\|_r^\theta\|\dot\varphi\|_q^{1-\theta}
		\end{aligned}
	\]
	where $c''_1\ge (c''_0)^\theta L^\frac{q-1}{q}$. Since $0\le\theta\le \zeta/(2-\zeta)$ and $0\le (q-1)/q\le \frac12$, we may put $c''_1=\max\{c''_0,(c''_0)^{\zeta/(2-\zeta)}\}\max\{1,L^{1/2}\}$.
	
	Thus, we have proved the case \cg for $c''=\hat c''=\max\{\hat c'+c_{SR},c''_0,c''_1\}$ and $\varphi\in C_0^\infty(0;1)$. In order to complete the proof of \cg, it remain to note that the space $C_0^\infty(0;1)$ is dense in $\mathring W^1_q(0;1)$ (see \cite[4.2.4]{Triebel}) and functionals $\int u_1\dot\varphi$, $\|\varphi\|_r$, and $\|\dot\varphi\|_q$ are all continuous on $\mathring W^1_q(0;1)$.
	
	Let us now proceed to the case \cw. On the one hand, since
	\[
		c'\lambda_0l(x)\|\dot\varphi\|_q =
		c'l(x)^{\frac{2q-1}{q-\zeta}}
		(\|\varphi\|_1/\|\dot\varphi\|_q)^{\frac{\zeta}{q-\zeta}},
	\]
	inequality $c'\,l(x)\lambda_0\|\dot\varphi\|_1\le 1$ is definitely fulfilled if
	\[
		c''_1 l(x)^{\frac{2q-1}{\zeta}}\|\varphi\|_1\le \|\dot\varphi\|_q
	\]
	where $c''_1=(c')^{-1+\frac{q}\zeta}$. On the other hand, if the opposite inequality is fulfilled for some test function $\varphi$, then $\|u_1\|_\infty\le l(x)\le L$ implies
	\[
		\begin{aligned}
			\int_0^1u_1\dot\varphi\,dt &\le
			l(x)\|\dot\varphi\|_1\le 
			l(x)\|\dot\varphi\|_q\le
			l(x)(c''_1l(x)^{\frac{2q-1}{\zeta}}
				\|\varphi\|_1)^\theta\|\dot\varphi\|_q^{1-\theta}=\\
			=&(c''_1)^\theta l(x)^{\kappa+2(q-1)}
				\|\varphi\|_r^\theta\|\dot\varphi\|_q^{1-\theta}=
			c_2''l(x)^\kappa
				\|\varphi\|_r^\theta\|\dot\varphi\|_q^{1-\theta}
		\end{aligned}
	\]
	where $c_2''=(c_1'')^\theta\max\{1,L^2\}$.
	
	Thus, in the case \cw, the desired inequality is proved under two additional restrictions: (i) $\varphi\in C_0^\infty(0;1)$ and (ii) $\varphi$ satisfies~\eqref{eq:main_varphi_restriction}. Note that equality~\eqref{eq:main_varphi_restriction} defines a closed linear subspace in $\mathring W^1_q(0;1)$ of the from
	\[
		H=\left\{
		\varphi\in \mathring W^1_{q}(0;1) : 
		\int_0^1 \varphi P_0([f_1,f_2]u_2+\ldots+[f_1,f_k]u_k)\,dt=0
		\right\}.
	\]
	Subspce $H$ has finite codimension $n$ as $u\in L_\infty(0;1)$ and $P_0\in L_\infty(0;1)$. Space $C_0^\infty(0;1)(0;1)$ is dense in $\mathring W^1_q(0;1)$ and, hence, it is dense in $H$. Therefore, the desired inequality is fulfilled for all $\varphi\in H$ as both functionals $\mathfrak{p}(\varphi)=c''_1 \|\varphi\|_2^{\theta}\|\dot\varphi\|_q^{1-\theta}$ and $\int_0^1u_1\dot\varphi\,dt$ are continuous on $\mathring W^1_q(0;1)$.
	
	Generalized derivative $\dot u_1$ can be considered as a linear functional on $\mathring W^1_q(0;1)$. Moreover, we have proved that it is dominated by $\mathfrak{p}(\varphi)$ on $H$. Hence, $\dot u_1|_H\le \mathrm{conv}\,\mathfrak{p}|_H$ due to linearity of $\dot u_1$. Functional $\mathrm{conv}\,\mathfrak{p}$ is positively homogeneous, convex, and nonnegative. Therefore, $\mathrm{conv}\,\mathfrak{p}$ is a seminorm on $\mathring W^1_q(0;1)$. Then, Hahn--Banach separation theorem implies that linear functional $\dot u_1$ has an extension~$\mathfrak{w}$ to the whole space $\mathring W^1_q(0;1)$ that is also dominated by $\mathrm{conv}\,\mathfrak{p}$, i.e.\ $\mathfrak{w}\le \mathrm{conv}\,\mathfrak{p}\le \mathfrak{p}$. Since $\mathfrak{w}$ and $\dot u_1$ coincide on $H$, the $H$ definition implies
	\[
		\dot u_1 = \mathfrak{w} + \langle\mu,P_0([f_1,f_2]u_2+\ldots+[f_1,f_k]u_k)\rangle
	\]
	for a tuple of $n$ constants $\mu\in\R^{n*}$. It remains to note that $P_0([f_1,f_2]u_2+\ldots+[f_1,f_k]u_k)\in L_\infty(0;1)$. So, for some constant $c''_3$ and any test function $\varphi\in\mathring W^1_q(0;1)$, we obtain
	\[
		\int_0^1\varphi \big\langle\mu,P_0([f_1,f_2]u_2+\ldots+[f_1,f_k]u_k)\big\rangle\,dt\le 
		c''_3\|\varphi\|_1\le
		c''_3\|\varphi\|_1^\alpha\|\dot\varphi\|_1^{1-\alpha}\le
		c''_3\|\varphi\|_2^\alpha\|\dot\varphi\|_q^{1-\alpha}.
	\]
	Thus, it is sufficient to put $\tilde c'' = c''_2 + c''_3 l(x)^{-\kappa}$ (recall that, in the case \cw, constant $\tilde c''$ may depend on the choice of local minimum $(x,u)$ in general).

\end{proof}

\subsection{Duality and pproximation of the optimal control}
\label{subsec:duality}

So, in order to improve length of the original trajectory $x(t)$, we are trying to choose a test function $\varphi(t)$ in such a way that the control variation $u_1\mapsto u_1-\lambda\dot\varphi$ improves the functional~\eqref{problem:sr_optimal_control}. It means that the term  $-u_1\dot\varphi$ should be as negative as possible. However, the $x(t)$ optimality forbids $-\int_0^1 u_1\dot\varphi\,dt$ to be too negative. Namely, it is bounded by Proposition~\ref{prop:interpolation_estimate}. Nonetheless, let us try to choose $\varphi(t)$ as best as possible assuming that norms $\|\varphi\|_r$ and $\|\dot\varphi\|_q$ in the right-hand side of~\eqref{eq:u_1_dot_phi_middle_estimate} are given. To do so, we consider the following convex optimization problem
\begin{equation}
	\label{problem:approximation_straight}
	\begin{gathered}
		-\int_0^1u_1\dot\varphi\,dt\to\min_{\varphi\in W^1_q[0;1]}\\
		\begin{cases}
		\|\varphi\|_r\le 1;\\
		\|\dot\varphi\|_q\le M;\\
		\varphi(0)=\varphi(1)=0.
	\end{cases}
	\end{gathered}
\end{equation}

Denote by $S_r^q(u_1,M)$ the value (i.e.\ infimum) of problem~\eqref{problem:approximation_straight}. So the $x(t)$ optimality implies the following lower bound on $S_r^q(u_1,M)$:

\begin{corollary}
\label{cor:SM_estimate}

	If for some $r\ge 1$, $q\ge 1$, $\zeta\le 1$, and $c''>0$, the control $u_1$ satisfies inequality~\eqref{eq:u_1_dot_phi_middle_estimate}, then, for any $M\ge 0$, one has
	\[
		S_r^q(u_1,M)\ge- c''l(x)^{\kappa} M^{1-\theta}.
	\]
\end{corollary}

\begin{proof}
	
	Indeed, \eqref{eq:u_1_dot_phi_middle_estimate} implies
	\[
		\int_0^1u_1\dot\varphi\,dt \le
		c''l(x)^\kappa\|\varphi\|_r^\theta\|\dot\varphi\|_q^{1-\theta}\le
		c''l(x)^\kappa M^{1-\theta}
	\]
	since $0\le\theta\le1$.
\end{proof}

The value $S_r^q(u_1,M)$ of infimum in problem~\eqref{problem:approximation_straight} is closely related to the approximation quality of the control $u_1$ by smooth function. It is well known that the rate of such an approximation is equivalent to an upper bound on $K$-functional of interpolation of the Lebesgue space $L_{q^*}(0;1)$ and Sobolev space $W^1_{r^*}(0;1)$. We will actively use this interpolation in the next subsection, and now we show duality of $K$-functional and value $S_r^q(u_1,M)$ of infimum in problem~\eqref{problem:approximation_straight}. Recall that $K$-functional is defined in our case as follows:
\begin{equation}
\label{eq:K_functional}
		K_{q^*}^{r^*}(M,u_1) \eqdef \inf_w(\|w\|_{W^1_{r^*}} + M\|w-u_1\|_{q^*}).
\end{equation}
where $\frac{1}{r}+\frac{1}{r^*}=\frac{1}{q}+\frac{1}{q^*}=1$.

In the following proposition, we use a functional that is very similar to $K_{q^*}^{r^*}$: we only change the term $\|w\|_{W^1_{r^*}}$ to the term $\|\dot w\|_{r^*}$. In fact, this change does not affect the interpolation procedure, which we explain in the next section.

\begin{proposition}[dual interpolational estimate on $K$-functional]
\label{prop:dual_problem}

	For any $u_1\in L_\infty(0;1)$, $1\le r<\infty$, $1\le q<\infty$, and $M\ge0$, the following equality holds\footnote{The second term is in fact $K$-functional up to some minor terms, and the first term is its dual interpolation estimate on control $u_1$}
	\begin{equation}
	\label{eq:u_1_approx}
		S_q^r(u_1,M) +
		\inf_{w\in W^1_{r^*}(0;1)}\big(
			\|\dot w\|_{r^*} + M\|u_1-w\|_{q^*}
		\big)= 0
	\end{equation}
\end{proposition}

\begin{proof}

	The statement trivially holds for $M=0$. So we assume $M>0$.
	
	Problem~\eqref{problem:approximation_straight} is convex. Hence, it is natural to apply duality theory of convex problems. In order to construct the dual problem, we should include problem~\eqref{problem:approximation_straight} into a family of convex problems. So let us take arbitrary $\psi=(\psi_0,\psi_1)\in L_r(0;1)\times L_q(0;1)$ and $\varphi^*\in W^1_{q^*}(0;1)$ and consider the following family of convex problems, which depend on parameters $\varphi^*$ and $\psi=(\psi_0,\psi_1)$:
	\begin{equation}
	\label{problem:straight_family}
		\begin{gathered}
			-\int_0^1(\dot\varphi^*\dot\varphi + \varphi^*\varphi)\,dt\to\min_{\varphi\in W^1_q(0;1)}\\
			\begin{cases}
				\|\varphi-\psi_0\|_r\le 1;\\
				\|\dot\varphi-\psi_1\|_q\le M;\\
				\varphi(0)=\varphi(1)=0.
		\end{cases}
		\end{gathered}
	\end{equation}
	Further, we show that problem~\eqref{problem:approximation_straight} is in family~\eqref{problem:straight_family} for $\psi_0=\psi_1=0$ and under an appropriate choice of parameter $\varphi^*$.
	
	Recall that a dual family to a family of optimization problems given in the following standard form (with parameters $\psi$ and $\varphi^*$)
	\begin{equation}
	\label{eq:straight_family_general}
		F(\varphi,\psi) - \langle \varphi^*,\varphi\rangle\to\min_\varphi
	\end{equation}
	is a family of optimization problems with the same parameters $\psi$ and $\varphi^*$ of the form
	\begin{equation}
	\label{eq:dual_family_general}
		F^*(\varphi^*,\psi^*) - \langle\psi,\psi^*\rangle\to\min_{\psi^*}
	\end{equation}
	where $F^*$ is the Legendre--Young--Fenchel convex conjugate to $F$. We emphasize that parameters $\varphi^*$ and $\psi$ are common for the both families, but original family problems are minimizing in $\varphi$, and dual family problems are minimizing in $\psi^*$.
	
	Usually, considerations of dual convex optimization problems require a careful choice of pair of spaces in duality. So we assume that $(\varphi,\varphi^*)$ and $(\psi,\psi^*)$ belong to some spaces $(\Phi,\Phi^*)$ and $(\Psi,\Psi^*)$, correspondingly. Moreover, we equip both pair of spaces $(\Phi,\Phi^*)$ and $(\Psi,\Psi^*)$ with the weak and weak$^*$ topologies, correspondingly (see~\cite{TikhomirovMagaril} for details). This careful choice is needed in the first place for proper computation of convex conjugates, e.g.\ $F^*$ and $F^{**}$.
	
	The dual family of optimization problems can be very useful as it is usually possible to prove that there is no duality gap between the origin family and the dual one. Namely, it means that the sum of infima in both families identically vanish for any choice of parameters $\psi$ and $\varphi^*$. Let us denote by $S_{\varphi^*}(\psi)$ the value of the corresponding problem in the original family (note that $\varphi^*$ enters the original family linearly, and $\psi$ enters it nonlinearly), and by $D_\psi(\varphi^*)$ the value of the corresponding problem in the dual family (here $\psi$ enters the dual family linearly, and $\varphi^*$ enters it nonlinearly). Then, a trivial inequality $F(\varphi,\psi) + F^*(\varphi,\psi)\ge \langle\varphi,\varphi^*\rangle + \langle\psi,\psi^*\rangle$ implies that $S_{\varphi^*}(\psi)+D_\psi(\varphi^*)\ge 0$. Moreover, we show below that there is no duality gap, i.e.\  $S_{\varphi^*}(\psi)+D_\psi(\varphi^*)=0$.
	
	\medskip
	
	Now we compute the dual family to the original one given in~\eqref{problem:straight_family}. In order to do so, let us rewrite family~\eqref{problem:straight_family} in the standard form. We have $\Phi=W^1_q(0;1)$, $\Phi^*=W^1_{q^*}(0;1)$, $\Psi=L_r[0;1]\times L_q(0;1)$, and $\Psi=L_{r^*}[0;1]\times L_{q^*}(0;1)$. We equip both spaces $\Phi$ and $\Psi$ with the weak topologies, and spaces $\Phi^*$ and $\Psi^*$ with the weak$^*$ topologies. Under this choice of topologies, pairs $(\Phi,\Phi^*)$ and $(\Psi,\Psi^*)$ become pairs of spaces in duality (see~\cite{TikhomirovMagaril}).
	
	Let us now write down a convenient formula for functional~$F$. Obviously, in order to throw out restrictions~\eqref{problem:straight_family} on $\varphi$ (which are formally not included into the standard form~\eqref{eq:straight_family_general}), we define $F$ as follows: $F(\varphi,\psi)=0$ if $\varphi$ does satisfy restrictions in~\eqref{problem:straight_family}, and $F(\varphi,\psi)=\infty$ if it doesn't:
	\[
		F(\varphi,\psi) = \begin{cases}
			0,&\text{if }\|\varphi-\psi_0\|_r\le 1,\ \|\dot\varphi-\psi_1\|_q\le M\text{ and } \varphi(0)=\varphi(1)=0;\\
			\infty,&\text{otherwise}.
		\end{cases}
	\]
	
	Before we proceed to computing the dual family, let us find a particular value of parameter $\varphi^*$ that includes the original problem~\eqref{problem:approximation_straight} into family~\eqref{problem:straight_family}. In other words, we are seeking for a function $\varphi^*=\hat\varphi^*$ such that
	\[
		\forall\varphi\in W^1_q[0;1]\qquad 
		\langle\hat\varphi^*,\varphi\rangle\eqdef
		\int_0^1(\varphi\hat\varphi^* + \dot\varphi\dot{\hat\varphi}^*)\,dt\equiv
		\int_0^1u_1\dot\varphi\,dt.
	\]
	Linear functional $\int_0^1u_1\dot\varphi\,dt$ is continuous on $W^1_q(0;1)$. Hence, the desired function $\hat\varphi^*$ exists and belongs to $(W^1_q(0;1))^*=W^1_{q^*}(0;1)$. So $\int_0^1(\dot\varphi(\dot{\hat\varphi}^*-u_1)+\varphi\hat\varphi^*)\,dt=0$ for any function $\varphi\in W^1_q(0;1)$. Therefore, the du Bois--Reymond lemma implies that $\dot{\hat\varphi}^*-u_1\in \mathring W^1_{q^*}(0;1)$ and
	\begin{equation}
	\label{eq:hat_varphi_star}
		\ddt(\dot{\hat\varphi}^*-u_1)=\hat\varphi^*.
	\end{equation}
	
	Under the described choice of $F$, original family~\eqref{problem:straight_family} is represented in standard form~\eqref{eq:straight_family_general}, and original problem~\eqref{problem:approximation_straight} is included into the family when $\psi=0$ and $\varphi^*=\hat\varphi^*$.
	
	Let us now compute the dual family (having the same parameters $\varphi^*$ and $\psi$). In order to do so, we first compute $F^*$:
	\[
		F^*(\varphi^*,\psi^*) = 
		\sup_{\varphi,\psi}\big(
			\langle\varphi^*,\varphi\rangle + \langle\psi^*,\psi\rangle - F(\varphi,\psi)
		\big) =
		\sup_{\substack{\|\varphi-\psi_0\|_r\le 1\\ \|\dot\varphi-\psi_1\|_q\le M\\\varphi(0)=\varphi(1)=0}}
		\int_0^1(\dot\varphi^*\dot\varphi + \varphi^*\varphi + \psi_0^*\psi_0 + \psi_1^*\psi_1)\,dt.
	\]
	Let us substitute $\eta_0=\psi_0-\varphi\in L_r(0;1)$ and $\eta_1=\psi_1-\dot\varphi\in L_q(0;1)$, or $\psi_0=\varphi+\eta_0$ and $\psi_1=\dot\varphi+\eta_1$:
	\[
		\begin{aligned}
			F^*(\varphi^*,\psi^*) = &
			\sup_{\substack{\|\eta_0\|_r\le 1\\ \|\eta_1\|_q\le M\\\varphi(0)=\varphi(1)=0}}
			\int_0^1(\dot\varphi^*\dot\varphi + \varphi^*\varphi + \psi_0^*(\varphi+\eta_0) + \psi_1^*(\dot\varphi+\eta_1))\,dt =\\
			=&\sup_{\|\eta_0\|_r\le 1} \int_0^1 \psi_0^*\eta_0\,dt + 
			\sup_{\|\eta_1\|_q\le M} \int_0^1 \psi_1^*\eta_1\,dt +
			\sup_{\varphi(0)=\varphi(1)=0} \int_0^1(\dot\varphi(\dot\varphi^*+\psi_1^*) + \varphi(\varphi^*+\psi_0^*))\,dt.
		\end{aligned}
	\]
	
	The first supremum in the right-hand side is equal to $\|\psi_0^*\|_{r^*}$, and the second one is equal to $M\|\psi_1^*\|_{q^*}$. So let us compute the third one. Terminal constraints $\varphi(0)=\varphi(1)=0$ are linear in $\varphi$. The inside integral is also linear in $\varphi$. Hence, there is an alternative: either the inside integral identically vanishes for all $\varphi\in W^1_q(0;1)$ having $\varphi(0)=\varphi(1)=0$ (and, in this case, the third supremum is equal to $0$) or there exists a function $\varphi$ having $\varphi(0)=\varphi(1)=0$ such that the inside integral does not vanish (and, in this case, the third supremum is equal to $\infty$). So, let us find all $\psi^*=(\psi_0^*,\psi_1^*)$ such that
	\[
		\int_0^1 (\dot\varphi(\dot\varphi^*+\psi_1^*) + \varphi(\varphi^*+\psi_0^*))\,dt = 0
	\]
	for all $\varphi$ having $\varphi(0)=\varphi(1)=0$. According to the du Bois--Reymond lemma, the latter is equivalent to the following: $\ddt(\dot\varphi^*+\psi_1^*)\in L_1(0;1)$ and $\ddt(\dot\varphi^*+\psi_1^*)=\varphi^*+\psi_0^*$, i.e.\ $\dot\varphi^*+\psi_1^*\in W^1_{r^*}(0;1)$.
	
	Thus, dual family~\eqref{eq:dual_family_general} of convex problems has the following form:
	\begin{equation}
	\label{problem:dual_family}
		\begin{gathered}
			\|\psi_0^*\|_{r^*} + M\|\psi_1^*\|_{q^*} - 
			\int_0^1(\psi_0\psi_0^*+\psi_1\psi_1^*)\to\min_{\psi^*\in L_{r^*}(0;1)\times L_{q^*}(0;1)}\\
			\ddt(\dot\varphi^*+\psi_1^*)=\varphi^*+\psi_0^*.
		\end{gathered}
	\end{equation}
	Recall that we denote by $D_\psi(\varphi^*)$ the value of the corresponding problem in the dual family.
	
	Functional $F$ is convex (as both norms $\|\cdot\|_r$ and $\|\cdot\|_q$ are convex, and terminal constrains $\varphi(0)=\varphi(1)=0$ are linear) and closed (as any convex lower semicontinuous in the strong topology functional is also lower semicontinuous in the weak topology\footnote{Indeed, any  strongly closed convex set is weakly closed by the Hahn--Banach theorem.}). Therefore, the Fenchel--Moreau theorem implies $F=F^{**}$. Hence, the bidual family (i.e.\ dual to dual family~\eqref{problem:dual_family}) coincides with the original one.
	
	\medskip
	
	Let us now show that there is no duality gap when parameters $\varphi^*$ and $\psi$ have the desired values, namely, $\psi=0$ and $\varphi^*=\hat\varphi^*$. So let us prove that $S_{\hat\varphi^*}(0)+D_0(\hat\varphi^*)=0$. Equalities $D_\psi(\varphi^*) + S_{\varphi^*}^{**}(\psi)=0$ and $S_{\varphi^*}(\psi) + D_\psi^{**}(\varphi^*)=0$ can be proved in standard way (see~\cite{TikhomirovMagaril}). Convexity of both functionals $F$ and $F^*$ implies convexity of $S_{\varphi^*}$ (on $\psi$) and $D_\psi$ (on $\varphi^*$). Therefore, the Fenchel--Moreau theorem implies $S_{\varphi^*}^{**} = \cl S_{\varphi^*}$ and $D_\psi^{**}=\cl D_\psi$. So it remains to show that functional $D_0$ is closed.
	
	Let us make a natural change of variables in~\eqref{problem:dual_family} and denote $v=\dot\varphi^*+\psi_1^*$. Then $v\in W^1_{r^*}(0;1)$ and problem~\eqref{problem:dual_family} for $\psi=0$ takes the following form:
	\[
		G(v)=\|\dot v-\varphi^*\|_{r^*} + M\|v-\dot\varphi^*\|_{q^*}\to\min_v.
	\]
	Clearly, the infimum value of $G$ does not exceed $\Delta=\|\varphi^*\|_{r^*} + M\|\dot\varphi^*\|_{q^*}$. We claim that function $G(v)$ is coercive. Indeed, if $\|\dot v-\varphi^*\|_{r^*}>\Delta$ or $\|v-\dot\varphi^*\|_{q^*}>\frac1M\Delta$, then $G(v)>\Delta$. Therefore, infimum of $G(v)$ over all functions $v$ coincides with that over the set of functions $v$ having $\|v-\dot\varphi^*\|_{q^*}\le\Delta$ and $\|\dot v-\varphi^*\|_{r^*}\le  \frac1M\Delta$. This set is convex, bounded, and strongly closed. Hence, it is weak$^*$ compact in $W^1_{r^*}(0;1)$ by Banach--Alaoglu theorem. Consequently, for any $\varphi^*$, there exists a global minimum~$v_{\varphi^*}$. Moreover, both norms $\|v_{\varphi^*}\|_{q^*}$ and $\|\dot v_{\varphi^*}\|_{r^*}$ are bounded as $G(v_{\varphi^*})\le\Delta$:
	\begin{equation}
	\label{eq:v_is_bounded}
		\|v_{\varphi^*}\|_{q^*} \le \frac1M\|\varphi^*\|_{r^*} + 2\|\dot\varphi^*\|_{q^*}
		\quad\text{and}\quad
		\|\dot v_{\varphi^*}\|_{r^*} \le 2\|\varphi^*\|_{r^*} + M\|\dot\varphi^*\|_{q^*}.
	\end{equation}
	
	Now we are able to show that $D_0$ is weak$^*$ lower semicontinuous using existence of a solution and bounds~\eqref{eq:v_is_bounded}. Let $\varphi_0^*\in W^1_{q^*}(0;1)$ be an arbitrary function. Consider a sequence $\varphi^*_m$ that converges to $\varphi^*_0$ in the weak$^*$ topology on $W^1_{q^*}(0;1)$ as $m\to\infty$ and such that the sequence $D_0(\varphi^*_m)$ has a limit. So we need to prove that $D_0(\varphi^*_0)\le\lim_{m\to\infty} D_0(\varphi^*_m)$.
	
	Since $\psi=0$, we have
	\[
		D_0(\varphi^*)=\inf_{\psi^*}F^*(\varphi^*,\psi^*) = G(v_{\varphi^*})
	\]
	as the infimum above is attained at  $\psi^*=(\dot v_{\varphi^*}-\varphi^*,v-\dot\varphi^*(\varphi^*))$ for any choice of $\varphi^*$. Consider the following sequence of solutions
	\[
		\psi^*_m=(\dot v_{\varphi^*_m}-\varphi^*_m,v_{\varphi^*_m}-\dot\varphi^*_m).
	\]
	Since $\varphi^*_m$  weakly$^*$ converges, the $\varphi_m^*$ norms are bounded by uniform boundedness principle. Therefore, \eqref{eq:v_is_bounded} implies that the $v_{\varphi^*_m}$ norms are bounded as $m\to\infty$. So using the Banach--Alaoglu theorem for the second time, we are able to find a subsequence of $v_{\varphi^*_m}$ that weakly$^*$ converges. Without loss of generality, we keep indices $m$ for the subsequence. Hence, sequence $\psi^*_m$ also weakly$^*$ converges to a function $\psi^*_0$ as $m\to\infty$. It remain to note that functional $F^*$ is lower weak$^*$ semicontinuous (as $F^*$ is convex conjugate to $F$), and so
	\[
		\lim_{m\to\infty} D_0(\varphi^*_m)=
		\lim_{m\to\infty}F^*(\varphi^*_m,\psi^*_m)\ge 
		F^*(\varphi^*_0,\psi^*_0)\ge D_0(\varphi_0^*).
	\]
	The latter proves that $D_0$ is weak$^*$ lower semicontinuous. Consequently, $S_{\varphi^*}(0)+D_0(\varphi^*)=0$ for all $\varphi^*$. In particular, it is true for $\varphi^*=\hat\varphi^*$.
	
	\medskip
	
	Now, as we have already shown absence of the duality gap, we can construct an estimate on the rate of $u_1$ approximation by functions from $W^1_{q^*}(0;1)$. In order to do so, let us simplify the dual family under the required choice of parameters $\psi=0$ and $\varphi^*=\hat\varphi^*$. According to~\eqref{eq:hat_varphi_star}, ODE in~\eqref{problem:dual_family} becomes $\ddt(u_1+\psi_1^*)=\psi_0^*$. Obviously, $u_1+\psi_1^*\in W^1_{r^*}(0;1)$. So, if $\psi=0$ and $\varphi^*=\hat\varphi^*$, then substitution $w=u_1+\psi_1^*$ gives the following equivalent problem:
	\[
		\|\dot w\|_{r^*} + M\|w-u_1\|_{q^*}\to \min_{w\in W^1_{r^*}}.
	\]
	Sum of that problem value and $S_r^q(u_1,M)$ is equal to 0 as there is no duality gap. So~\eqref{eq:u_1_approx} is proved.

\end{proof}

\subsection{Interpolation}
\label{sec:interpolation}

Duality proved in Proposition~\ref{prop:dual_problem} together with the interpolational estimate proved in Proposition~\ref{prop:interpolation_estimate} and Corollary~\ref{cor:SM_estimate} imply the following estimate on the rate of control $u_1$ approximation by smooth functions. Namely, we have proved that, for any $M\ge 0$, there exists a function $w\in W^1_{r^*}(0;1)$ such that
\begin{equation}
\label{problem:dual}
	\|\dot w\|_{r^*} + M\|w-u_1\|_{q^*} \le 2c''l(x)^\kappa M^{1-\theta}
\end{equation}
where $q^*=q/(q-1)\ge 2$, $r^*=2$ in the case \cw, and $r^*=\infty$ in the case \cg, and $\theta$ and $\kappa$ are defined by $\zeta$ in~\eqref{eq:u_1_dot_phi_middle_estimate}. According to Proposition~\ref{prop:interpolation_estimate}, $\zeta=\frac2s$ in the case \cw and $\zeta=\frac1s$ in the case \cg. 

Compare lhs of~\eqref{problem:dual} with $K$-functional given in~\eqref{eq:K_functional}. So we need to estimate norm $\|w\|_{W^1_{r^*}(0;1)}$. There are a lot of different pairwise equivalent norms on space $W^1_{r^*}(0;1)$. For our purposes, it is convenient to use $\|w\|_{W^1_{r^*}(0;1)}=\|\dot w\|_{r^*} + \|w\|_{q^*}$. The first term is trivially estimated by~\eqref{problem:dual}. In order to estimate $\|w\|_{q^*}$, let us use inequality $|u_1|\le l(x)$: in the case \cg, we have $0<\kappa\le 1$ and so $l(x)\le l(x)^\kappa L^{1-\kappa}$; and in the case \cw, we have $l(x)=l(x)^\kappa L^{1-\kappa}$. Therefore, if $M>1$, then
\[
	\|w\|_{q^*} \le (2c''+L^{1-\kappa})l(x)^\kappa M^{-\theta}\le (2c''+L^{1-\kappa})l(x)^\kappa M^{1-\theta}
\]
and
\[
	\|\dot w\|_{W^1_{r^*}} + M\|w-u_1\|_{q^*}\le (4c''+L^{1-\kappa})l(x)^\kappa M^{1-\theta}.
\]
If the opposite $0\le M\le 1$ holds, then we can always take $w=0$, which guarantees
\[
	\|\dot w\|_{W^1_{r^*}} + M\|w-u_1\|_{q^*} = M\|u_1\|_{q^*}\le l(x)M\le  L^{1-\kappa}l(x)^\kappa M^{1-\theta}.
\]
Hence, $K$-functional of interpolation satisfies the following estimate:
\[
	K_{q^*}^{r^*}(M,u_1) = \inf_w(\|w\|_{W^1_{r^*}} + M\|w-u_1\|_{q^*})\le c'''l(x)^\kappa M^{1-\theta}
\]
where $c''' \ge 4c''+L^{1-\kappa}$. Since $\kappa$ is between $1$ and $(1+\zeta)/(2-\zeta)$, we can put $c''' = 4c''+\max\{1,L^{(1-2\zeta)/(2-\zeta)}\}$.

Thus, the control $u_1$ belongs to the following interpolation space (see \cite[section 1.3.1]{Triebel}):
\begin{equation}
\label{eq:u_1_in_interpolation_space}
	u_1\in X=(W^1_{r^*}(0;1),L_{q^*}(0;1))_{1-\theta,\infty}=(L_{q^*}(0;1),W^1_{r^*}(0;1))_{\theta,\infty}.
\end{equation}
So, it remains to reduce the interpolation space in the right-hand side to a standard one.

Interpolation~\eqref{eq:u_1_in_interpolation_space} belongs to a so called nondiagonal case as $\frac{1-\theta}{q^*} + \frac{\theta}{r^*}\ne \frac1\infty$. Unfortunately, no exact formula for the nondiagonal interpolation is known. Nonetheless, we can exchange the nondiagonal interpolation space $X$ by a little bit bigger space, which we are actually able to compute explicitly. So let us consider the cases \cg and \cw separately.

\medskip

In the case \cg, we have $r^*=\infty$, $q^*\ge 2$, and $\theta$ is given in~\eqref{eq:u_1_dot_phi_middle_estimate}. Hence, $X$ defined in~\eqref{eq:u_1_in_interpolation_space} becomes
\[
	X=(L_{q^*}(0;1),W^1_{\infty}(0;1))_{\theta,\infty}.
\]
Let us use embedding theorems of Sobolev type (see \cite[4.6.2]{Triebel}). First, $L_{q^*}(0;1)=H^0_{q^*}(0;1)\subset B^0_{q^*,q^*}(0;1)$ as $q^*\ge 2$. Second, we fix $q^*< a<\infty$ and use embedding $W^1_{\infty}(0;1)\subset B^1_{a,a}(0;1)$. Obviously, $B^1_{a,a}(0;1)\subset B^0_{q^*,q^*}(0;1)$. Hence, for any $1\le p\le\infty$ and $0<\beta<\theta$, we obtain (see \cite[1.3.3]{Triebel})
\[
	X\subset (B^0_{q^*,q^*}(0;1),B^1_{a,a}(0;1))_{\theta,\infty} \subset (B^0_{q^*,q^*}(0;1),B^1_{a,a}(0;1))_{\beta,p}.
\]
Let us make a choice of $p$ that leads to the diagonal type interpolation, namely
\[
	\frac1p=\frac{1-\beta}{q^*} + \frac{\beta}{a}.
\]
Then, the corresponding theorem for the diagonal type interpolation (see \cite[4.3.1]{Triebel}) gives
\[
	X\subset (B^0_{q^*,q^*}(0;1),B^1_{a,a}(0;1))_{\beta,p} = 
	B^{\beta}_{p,p}(0;1)\subset B^{\beta}_{p,\infty}(0;1).
\]

Now, we find values that parameter $\beta$ can take assuming that a fixed value of $p$ is given. So the following restrictions must be fulfilled: $a>q^*\ge 2$ and
\[
	0<\beta=\frac{\frac1{q^*}-\frac1p}{\frac1{q^*}-\frac1a}<
	\theta=\frac{(q-1)\zeta}{q-\zeta}=\frac{\zeta}{q^*(1-\zeta)+\zeta}=\frac{1}{q^*(s-1)+1}.
\]
The latter equality holds as $\zeta=\frac1s$ in the case \cg. Since $a$ can be taken arbitrary large, the previous pair of inequalities is equivalent to the following:
\[
	0<1-\frac{q^*}{p} < \frac{1}{q^*(s-1)+1}
	\qquad\Leftrightarrow\qquad
	p-\frac{1}{s-1}<q^*<p.
\]
Consequently, if $p-\frac1{s-1}\ge 2$, then $q^*$ can be taken arbitrary close to $p-\frac1{s-1}$. So, proceeding to the limit as $a\to\infty$ and $q^*\to p-\frac1{s-1}$, we obtain $\theta\to \frac{1}{p(s-1)}$. Note that if $q^*\to p-\frac1{s-1}$, then $\kappa\to 1-\frac{s-2}{p(s-1)}$. Since $B^{\beta}_{p,\infty}\subset B^{\beta-\varepsilon}_{p,\infty}$ for an arbitrary $\varepsilon>0$, number $\kappa$ can be taken arbitrary close to $1-\frac{s-2}{p(s-1)}$ independently of the choice of $\beta$. Hence, for any $p\ge 2+\frac1{s-1}$, $0<\beta<\frac{1}{p(s-1)}$, and $0<\kappa<1-\frac{s-2}{p(s-1)}$, there exists a constant $C>0$ such that
\[
	\|u_1\|_{B^\beta_{p,\infty}(0;1)}\le C l(x)^\kappa.
\]

Now, let us enlarge the interval of possible $\kappa$ values using embeddings of Sobolev type. Let us find $\tilde p>p$ such that we are still able to select a number $\tilde\beta<\frac{1}{\tilde p(s-1)}$ such that the embedding $B^{\tilde\beta}_{\tilde p,\infty}(0;1)\subset B^{\beta}_{p,\infty}(0;1)$ holds. Obviously, this embedding holds if the following conditions are satisfied (see \cite[4.6.2]{Triebel})
\[
	\beta<\tilde\beta
	\quad\text{and}\quad
	\beta-\frac1p < \tilde \beta - \frac1{\tilde p}.
\]
Since $\tilde p>p$, the second inequality follows from the first one. Thus, we need to satisfy only $\beta<\tilde\beta<\frac{1}{\tilde p(s-1)}$. Therefore, existence of the required $\tilde\beta$ is guaranteed if $\tilde p>\frac{1}{\beta(s-1)}$. Note that $\frac{1}{\beta(s-1)}>p$ as $\beta<\frac{1}{p(s-1)}$.

Thus, for some constant $C'>0$, we have
\[
	\|u_1\|_{B^\beta_{p,\infty}(0;1)}\le C'\|u_1\|_{B^{\tilde\beta}_{\tilde p,\infty}(0;1)}.
\]
The last norm (as we have already shown) can be estimated by $\tilde Cl(x)^{\tilde\kappa}$ where $\tilde\kappa<1-\frac{s-2}{\tilde p(s-1)}$. Since $\tilde p$ can be taken arbitrary close to $\frac{1}{\beta(s-1)}$, $\tilde\kappa$ can take values that are arbitrary close to $1-\beta(s-2)$. In other words, for any $\tilde\kappa<1-\beta(s-2)$, there exists a constant $C''=C'\tilde C$ such that
\[
	\|u_1\|_{B^\beta_{p,\infty}(0;1)} \le C'' l(x)^{\tilde\kappa},
\]
and item \cg in Theorem~\ref{thm:main} is proved.

\medskip

Now we proceed to the case \cw. In that case, we have $r^*=2$. So space $X$ defined in~\eqref{eq:u_1_in_interpolation_space} takes the form
\[
	X=(L_{q^*}(0;1),W^1_2(0;1))_{\theta,\infty}.
\]

Since $W^1_2(0;1)=B^1_{2,2}(0;1)$ and $L_{q^*}(0;1)\subset B^0_{q^*,q^*}(0;1)$ (see \cite[4.6.2]{Triebel}), we obtain
\[
	X\subset (B^0_{q^*,q^*}(0;1),B^1_{2,2}(0;1))_{\theta,\infty}.
\]
Again, $B^1_{2,2}(0;1)\subset B^0_{q^*,q^*}(0;1)$ as $1-\frac12 > 0-\frac1{q^*}$ (see \cite[4.6.2]{Triebel}). Therefore, for any $0<\alpha<\theta$ and $1\le p\le\infty$, we have (see \cite[1.3.3]{Triebel})
\[
	X\subset (B^0_{q^*,q^*}(0;1),B^1_{2,2}(0;1))_{\alpha,p}.
\]
If we take $p$ that satisfies $\frac{1-\alpha}{q^*} + \frac{\alpha}{2}=\frac1p$, then we obtain the following diagonal type interpolation:
\[
	(B^0_{q^*,q^*}(0;1),B^1_{2,2}(0;1))_{\alpha,p} = B^{\alpha}_{p,p}(0;1) \subset B^{\alpha}_{p,\infty}(0;1).
\]

We emphasize that $p=((q^*)^{-1} + \alpha(\frac12-(q^*)^{-1}))^{-1}\to((q^*)^{-1} + \theta(\frac12-(q^*)^{-1}))^{-1}+0$ as $\alpha\to \theta-0$ since $q^*\ge 2$. Therefore $p\ge 2$. Further, as before, we fix $p\ge 2$ and check values that parameter $\alpha$ can take:
\[
	0<\alpha=\frac{\frac1p-\frac1{q^*}}{\frac12-\frac1{q^*}} < \theta = \frac{\zeta}{q^*(1-\zeta)+\zeta}=\frac{2}{q^*(s-2)+2}
\]
The latter equality holds as $\zeta=\frac2s$ in the case \cw. These inequalities are equivalent to the following ones:
\[
	p<q^*<\frac{p(s-1)-2}{s-2}
\]
This pair of inequalities is compatible if $p>2$. Thus, for any $p>2$, number $q^*$ can be taken arbitrary close to the right-hand side, and $\theta$ becomes arbitrary close to $\frac{2}{p(s-1)}$. Therefore,
\begin{equation}
\label{final:u1_embedding_local}
	u_1\in B^\alpha_{p,\infty}(0;1)
\end{equation}
for any $p>2$ and $\alpha<\frac{2}{p(s-1)}$. So we have proved item \cw in Theorem~\ref{thm:main} for $p>2$. It remains to consider the case $p=2$. If $p=2$, then for $q^*=2$, $r^*=2$, and $\theta=\frac{1}{s-1}$ (see~\eqref{eq:u_1_dot_phi_middle_estimate} for $\zeta=2/s$), definition \eqref{eq:u_1_in_interpolation_space} gives
\[
	X=(L_2(0;1),W^1_2(0,1))_{\frac1{s-1},\infty}.
\]

Since $L_2(0;1)=H^0_2(0;1)$ and $W^1_2(0;1)=H^1_2(0;2)$ (see \cite[sections 2.3.1 and 4.2.1]{Triebel}), this interpolation belongs to the diagonal case, and according to~\cite[section 4.3.1]{Triebel} we have
\[
	X=(H^0_2(0;1),H^1_2(0,1))_{\frac1{s-1},\infty} = B^{\frac1{s-1}}_{2,\infty}(0;1).
\]
It remains to note that $B^{\frac1{s-1}}_{2,\infty}(0;1)\subset B^\alpha_{2,\infty}(0;1)$ for any $\alpha<\frac1{s-1}$ (see~\cite[section 4.6.2]{Triebel}).

\end{document}